\documentclass[10pt, a4paper]{article}
\usepackage{dsfont}

\usepackage{graphicx}
\usepackage{amsmath,amssymb,array}
\usepackage{amsthm} 
\usepackage[ruled,vlined]{algorithm2e}
\usepackage{graphics,graphicx}
\usepackage{multirow}
\newcolumntype{P}[1]{>{\centering\arraybackslash}p{#1}}

\newtheorem{theorem}{Theorem}
\newtheorem{lemma}{Lemma}

\newtheorem{definition}{Definition}
\newtheorem{remark}{Remark}

\usepackage[font=small]{caption}

\usepackage{dsfont}
\newcommand{\R}{\mathds{R}}

\newcommand{\0}{\mathbf{0}}
\newcommand{\1}{\mathbf{1}}

\newcommand{\ee}{\mathbf{e}}

\newcommand{\sve}{\mathbf{s}}
\newcommand{\xx}{\mathbf{x}}
\newcommand{\yy}{\mathbf{y}}
\newcommand{\uu}{\mathbf{u}}
\newcommand{\vv}{\mathbf{v}}
\newcommand{\ww}{\mathbf{w}}
\newcommand{\zz}{\mathbf{z}}
\newcommand{\Tr}{\mathrm{Tr}}
\newcommand{\kappaE}{\ensuremath{\kappa_{\rm E}}}
\newcommand{\kappaT}{\ensuremath{\kappa_{\rm T}}}

\newcommand{\E}{\ensuremath{\mathbb{E}}}

\newcommand{\cov}{\mathop{\rm cov}}
\newcommand{\argmin}{\mathop{\rm arg\,min}}

\newcommand{\Sym}{\mathop{\rm SYM}}

\newcommand{\Pd}{\mathop{\rm PD}}

\newcommand{\ls}{\ensuremath{\texttt{\textsc{LS}}}}
\newcommand{\als}{\ensuremath{\texttt{\textsc{ALS}}}}
\usepackage{color}

\providecommand{\abs}[1]{\left\lvert#1\right\rvert}
\providecommand{\norm}[1]{\left\lVert#1\right\rVert}
\newcommand{\lin}[1]{\ensuremath{\left\langle #1 \right\rangle}}
\newcommand{\normal}{\mathcal{N}}
\DeclareMathOperator{\expect}{\mathbb{E}}

\newcommand{\reals}{\mathds{R}}

\begin{document}
\title{Variable Metric Random Pursuit\footnotemark[1]}  
\author{Sebastian U. Stich\footnotemark[2] \and C.~L. M\"{u}ller\footnotemark[4] \and Bernd G\"{a}rtner\footnotemark[3]}

\footnotetext[1]{The project CG Learning acknowledges the financial support of the Future and Emerging Technologies (FET) programme within the Seventh Framework Programme for Research of the European Commission, under FET-Open grant number: 255827.}
\footnotetext[2]{Institute of Theoretical Computer Science, ETH Z\"urich, \texttt{stich@inf.ethz.ch}}
\footnotetext[4]{Institute of Theoretical Computer Science, ETH Z\"urich, \texttt{cm192@nyu.edu}}
\footnotetext[3]{Institute of Theoretical Computer Science, ETH Z\"urich, \texttt{gaertner@inf.ethz.ch}}

\date{September 30, 2014}

\maketitle



\begin{abstract}
We consider unconstrained randomized optimization of smooth convex objective functions in the gradient-free setting.
We analyze Random Pursuit (RP) algorithms with fixed (F-RP) and variable metric (V-RP). The algorithms only use zeroth-order information about the objective function and compute an approximate solution by repeated optimization over randomly chosen one-dimensional subspaces. The distribution of search directions is dictated by the chosen metric. 
Variable Metric RP uses novel variants of a randomized zeroth-order Hessian approximation scheme recently introduced by Leventhal and Lewis (D. Leventhal and A. S. Lewis., Optimization 60(3), 329--245, 2011). We here present (i) a refined analysis of the expected single step progress of RP algorithms and their global convergence on (strictly) convex functions and (ii) novel convergence bounds for V-RP on strongly convex functions. We also quantify how well the employed metric needs to match the local geometry of the function in order for the RP algorithms to converge with the best possible rate. 
Our theoretical results are accompanied by numerical experiments, comparing V-RP with the derivative-free schemes CMA-ES, Implicit Filtering, Nelder-Mead, NEWUOA, Pattern-Search and Nesterov's gradient-free algorithms.

\end{abstract}

\section{Introduction}
Since its inception by Davidon in the late 1950's \cite{Davidon:1991} variable metric methods have become a cornerstone in first-order (non-)convex continuous optimization. Among the many instances of variable metric schemes Quasi-Newton methods such as the BFGS scheme \cite{Broyden:1970,Fletcher:1970,Goldfarb:1970,Shanno:1970} are ubiquitous in all areas of science and engineering. In zeroth-order (or gradient-free) optimization, the idea of using a variable metric guiding the search for local or global optima has surprisingly been used to a far less extent. Although ``directional adaptation" has been conjectured to be useful for randomized gradient-free schemes in the late 1960's~\cite{Schumer:1968} the literature on this topic is scarce and scattered across different communities ranging from electrical engineering, optimal control, bio-inspired optimization to mathematical programming. Important examples include the Gaussian Adaptation algorithm developed by Kjellstr\"om and Taxen~\cite{Kjellstrom:1981,Mueller:2010a} in the context of analog circuit design, Marti's controlled random search schemes using concepts from optimal control \cite{Marti:1986}, and the arguably most popular scheme, Hansen's Evolution Strategy with Covariance Matrix Adaptation (CMA-ES)~\cite{Hansen:2001} that emerged in the bio-inspired optimization community. 

Despite their great appeal in practice many randomized gradient-free variable metric schemes lack a thorough theoretical convergence analysis. A marked exception is Leventhal and Lewis' recent work on Randomized Hessian approximation (RHE)~\cite{Leventhal:2011}. We here adopt some of their ideas and extend our framework of Random Pursuit (RP)~\cite{Stich:2011}, eventually leading to Variable Metric Random Pursuit (V-RP) schemes. We solely consider optimization problems of the kind:
\begin{align}
 \label{eq:problem}
 \min f(\xx) \quad \text{subject to} \quad \xx \in \R^n \, , 
\end{align}
where $f$ is a smooth convex function. We assume that there is a
global minimum and that the curvature of the function $f$ is bounded from above. Moreover, we assume 
that we have only access to function values of $f$. No analytic gradient or higher order information 
about $f$ is available.

To motivate Variable Metric Random Pursuit, let us first sketch the working mechanism 
of standard Random Pursuit on an illustrative example. Each iteration of standard 
Random Pursuit consists of two steps: (i) a random direction is sampled from an 
isotropic probability distribution; (ii) the next iterate is chosen such as to (approximately) 
minimize the objective function along this direction.
In~\cite{Stich:2011} we have shown that the expected error in function
value decreases by a factor of
$\left(1-\frac{m}{n\ell}\right)$
in every step, if $m>0$ and $\ell>0$ are parameters of quadratic
functions that bound the difference between $f$ and any of its linear
approximations from below and above\footnote{The left inequality is usually referred to as strong-convexity; the right one follows from Lipschitz continuity of the gradient. See Section~\ref{sec:notation}.}; more precisely,
\begin{equation}
\label{eq:standard_rp}
\frac{m}{2}\norm{\yy-\xx}^2 \leq \ell_{\xx}(\yy):= f(\yy)-f(\xx)-\lin{\nabla f(\xx),\yy-\xx}\leq
\frac{\ell}{2}\norm{\yy-\xx}^2
\end{equation}
is assumed to hold for all $\xx,\yy$. For twice differentiable functions this condition is equivalent to an uniform lower and upper bound on the Hessian $H(\xx)$: $m \preceq H(\xx) \preceq \ell$.
As an example, let us consider the function 
\[
f_0(x_1,x_2) = 100x_1^2 + x_2^2,
\]
for which $\ell_{\xx}(\yy) = 100(x_1-y_1)^2 + (x_2-y_2)^2$. This means
that $m=2$ and $\ell=200$ are the best possible parameters in
(\ref{eq:standard_rp}), and the progress rate  in every step is no
better than $(1-1/200)$. This also matches our intuition: every
level set of $f_0$ is a long and skinny ellipse, stretching out along
the $x_2$-axis; if we start from a point close to the $x_2$ axis, the
progress in a step will be small, unless we almost sample in $x_2$-direction.

For this particular function $f_0$, it would be better to sample from an 
anisotropic distribution that favors the $x_2$-direction. Once we fix such 
an anisotropic sampling distribution, however, other functions become ``bad''; 
in fact, without prior knowledge about $f$,
anisotropic sampling makes no sense at all. Here is where 
the ``variable metric'' approach comes in. The idea is
to gradually \emph{adapt} the sampling distribution to the function
$f$ while we run the algorithm. Suppose that we can somehow estimate
the Hessians at the various iterates. Under the assumption that $f$ is
wedged between two quadratic functions---whose Hessians are not
necessarily multiples of the identity, as in~(\ref{eq:standard_rp})---these estimates will allow us to learn a
suitable metric that guides the sampling distribution. 
In case of $f_0$, we would start with the isotropic one and then converge 
to a distribution that indeed favors the $x_2$-direction with the right proportion.

In this contribution we present a framework for analyzing 
the convergence behavior of Random Pursuit algorithms on convex functions.
In a first step we analyze the Fixed Metric Random Pursuit (F-RP)
algorithm
for fixed (anisotropic) 
sampling distributions. 
%
In a second step we equip Random Pursuit with a randomized
scheme to update the metric that defines the sampling distribution in every step: the Variable Metric 
Random Pursuit. We present precise theoretical analysis of an update scheme recently proposed 
by Leventhal and Lewis \cite{Leventhal:2011} as well as three novel implementations.. These learning schemes are generic in the 
sense that they work for all convex functions and do not require any prior
knowledge of the function's shape. We prove that the 
sampling distribution converges to a distribution that yields 
asymptotically optimal (and function-independent) progress rates.
The proposed schemes are easily parallelizable, thus allowing a computational speed-up of the update schemes on multi-core machines. 

The remainder of the paper is structured as follows. In Section \ref{sec:rp} we give a generic description
of the different Random Pursuit algorithms and their essential building blocks. We introduce all relevant 
mathematical definitions such as matrix upper and lower bounds of convex functions and expressions 
for certain scalar and matrix expectations in Section \ref{sec:notation}. We derive the expected 
single-step progress and global convergence of F-RP in Section \ref{sec:convergence}. 
Section \ref{sec:metriclearning} is dedicated to Variable Metric Random Pursuit. 
We discuss the key results of  
the paper and outline future research goals in Section \ref{sec:discussion}.

\section{Fixed and Variable Metric Random Pursuit}
\label{sec:rp}
All Random Pursuit algorithms are designed for problems 
as in~\eqref{eq:problem}.
Before stating the formal definition of the considered RP algorithms we need to define 
one indispensable primitive.
\begin{definition}[Line search oracle]
\label{def:lineseach}
For $\xx \in \R^n$, a direction $\uu \in \R^n \setminus \{\0\}$ and a convex function $f$, 
a function 
$\ls_f \colon \reals^n \times \R^n \to \reals$ with
%
\begin{align}
\ls_f(\xx,\uu) = \displaystyle \argmin_{h \in \reals} f(\xx + h\uu)   \label{def:exactlinesearch}
\end{align}
is called an \emph{exact line search oracle}. 
\end{definition}

The two RP schemes considered here are summarized in Fig.~\ref{fig:algos}. 
\def \sfactor{0.9}
\begin{figure}[!b]
\centering
\null\hfill %
\scalebox{\sfactor}{
	\begin{minipage}[t]{.51\linewidth}
	     \vspace{0pt} 
	      \begin{algorithm}[H]
	     %
	     %
	      \NoCaptionOfAlgo
	      \SetKwInOut{Output}{Output}
	      \Output{Approximate solution $x_N$ to \eqref{eq:problem}}
	      \SetAlgoVlined
	      \DontPrintSemicolon
	      \SetInd{0.3em}{0.5em}
	      \nl \For{$k=1$ to $N$}{
	      \nl $\uu_{k} \sim \mathcal{N}(\0, \Sigma)$\;
	      \nl $\xx_{k} \leftarrow \ls_f(\xx_{k-1}, \uu_{k})$\;
	      }%
	      \nl \Return{$\xx_N$}\;
	      \caption{{$\texttt{F-RP}(f, \xx_0, \Sigma, N)$}}%
	     \end{algorithm}%
	\end{minipage}  
}%
\null\hfill %
\scalebox{\sfactor}{
	\begin{minipage}[t]{.51\linewidth}
		     \vspace{0pt} 
	      \begin{algorithm}[H]
	     %
	     %
	      \NoCaptionOfAlgo
	      \SetKwInOut{Output}{Output}
	      \Output{Approximate solution $x_N$ to \eqref{eq:problem}}
	      \SetAlgoVlined
	      \DontPrintSemicolon
	      \SetInd{0.3em}{0.5em}
	      \nl \For{$k=1$ to $N$}{
	      \nl $B_k \leftarrow \texttt{updateHess}(f, \xx, B_{k-1})$
	      \nl $\uu_{k} \sim \mathcal{N}(\0, B_k^{-1})$\;
	      \nl $\xx_{k} \leftarrow \ls_f(\xx_{k-1}, \uu_{k})$\;
	      }%
	      \nl \Return{$\xx_N$}\;
	      \caption{{$\texttt{V-RP}(f, \xx_0, B_0, N)$}}%
	     \end{algorithm}%
	\end{minipage}
}%
\hfill\null %
\caption{Fixed Metric Random Pursuit (left panel) and the Variable Metric version (right panel). 
The generic sub-routine $\texttt{updateHess}$ on line 2 exemplifies any function that 
generates the metric $B_k$ in step $k$. Three specific instantiations are discussed 
in Sec.~\ref{sec:metriclearning} (cf. Fig.~\ref{fig:algonew}). \label{fig:algos}}
\end{figure}
In Fixed Metric Random Pursuit (F-RP) a direction $\uu \in \R^n \setminus \{\0\}$ is sampled 
from a multivariate normal distribution with fixed covariance $\Sigma$ at iteration $k$ of 
the algorithm. The next iterate $\xx_{k}$ is calculated from the current iterate $\xx_{k-1}$ as
\begin{align}
\xx_{k} := \xx_{k-1} + \ls_f(\xx_{k-1},\uu) \cdot \uu \,. \label{eq:rponestep}
\end{align}
This algorithm only requires function evaluations in addition to the
line search oracle. No  first or second-order information about the objective is needed. We emphasize that besides the starting point no further input parameters describing function properties (such as curvature constant etc.) are necessary.  
The actual run time will, however, depend on the specific properties of the objective function and on the choice of the covariance matrix $\Sigma$, as detailed in Section~\ref{sec:convergence}. 
%
%
Variable Metric Random Pursuit (V-RP) comprises an independent process that gives 
an approximation of the Hessian at each iteration. The inverse of the Hessian is then 
used as covariance matrix in the multivariate normal distribution to generate the current 
search direction. In principle, any deterministic or randomized gradient-free estimator can 
be used for this purpose. 
In Section \ref{sec:metriclearning} we will use a Randomized Hessian approximation scheme recently proposed in \cite{Leventhal:2011} for this task.


For simplicity we assumed here access to an exact line search oracle. However, approximate line search schemes are sufficient to establish convergence of the Random Pursuit algorithms. We introduce such oracles in Section~\ref{sec:commentls}.

%

%
%
\section{Definitions and Notations}
\label{sec:notation}
We now introduce the  notation and some inequalities that will be useful for the subsequent analysis. Most importantly, we define two classes of convex functions with respect to so-called quadratic norms. This extends the standard model 
and allows us to derive convergence rates that take the eigenvalue spectrum of the Hessian into account.
\subsection{Quadratic norms}
\label{sec:quadnorms}
Let $\Pd_n$ denote the set of symmetric positive definite $n \times n$
matrices. With respect to $A \in \Pd_n$, we can define an 'anisotropic' norm by
 $\norm{\xx}_A^2 := \lin{\xx,\xx}_A$
for $\xx, \yy \in \R^n$. 
In statistics this metric is also known as the Mahalanobis metric. We observe that
\begin{align}
\label{eq:normbound}
\lambda_{\rm min}(A) \norm{\xx}^2 
\leq \norm{\xx}_A^2 \leq \lambda_{\rm max}(A) \norm{\xx}^2\,,
\end{align}
due to $\lambda_{\rm min}(A)=\min\{\xx^TA\xx: \|x\|=1\}$ and
$\lambda_{\rm max}(A)=\max\{\xx^TA\xx: \|x\|=1\}$.
This statement can be generalized, as shown in the following lemma.
\begin{lemma}
\label{lem:AvsBnorm}
Let $A,B\in\Pd_n$ and $\xx\in\R^n$. Then
\begin{align}
\lambda_{\rm min}(B^{-1}A) \norm{\xx}_B^2 \leq \norm{\xx}_{A}^2\leq \lambda_{\rm max}(B^{-1}A) \norm{\xx}_B^2\,. \label{eq:normbound2}
\end{align}
\end{lemma}
A proof can be found in~\cite[Prop.~18.3]{Puntanen:2011}.
The substitution $\xx = (B^{1/2})^{-1} \yy$, where $B^{1/2} \in \Pd_n$ denotes the positive semidefinite root of $B$ and $\yy \in \R^n$, allows to reduce~\eqref{eq:normbound2} to~\eqref{eq:normbound}. 
It remains to note that $(B^{1/2})^{-1}A(B^{1/2})^{-1}$ and
$AB^{-1}$ have the same eigenvalues, for this see e.g. again~\cite[Prop.~13.2]{Puntanen:2011}. 
%
%
%
\subsection{Quadratic bounds}
\label{sec:qbounds}
We now define two function classes. We assume that the objective function $f$ in~\eqref{eq:problem} is differentiable and convex. 
The latter property is equivalent to
\begin{align}
\label{def:convex1}
f(\yy) \geq f(\xx) + \lin{\nabla f(\xx), \yy-\xx}, \quad \xx, \yy \in \R^n \, .
\end{align}
We also require that the curvature of $f$ is bounded. However, we allow for 
different curvatures depending on the direction. By this we mean that 
for some fixed symmetric and positive definite matrix $L \in 
\Pd_n$, 
\begin{align}
\label{def:lipschitz}
f(\yy)-f(\xx) - \lin{\nabla f(\xx), \yy-\xx} \leq \frac{1}{2} \norm{\xx-\yy}_{L}^2,%
 \quad \xx,\yy \in \R^n \, .
\end{align}
We will also refer to this inequality as the \emph{(matrix) quadratic 
upper bound}. 
We denote by $C_{L}^1$ the class of (once) differentiable convex functions for which~\eqref{def:lipschitz}
holds with parameter $L$.
%
%
A differentiable function is \emph{strongly convex} with parameter $M 
\in \Pd_n$ if the \emph{(matrix)  quadratic lower bound} 
\begin{align}
\label{def:stronconvex}
f(\yy) - f(\xx) - \lin{\nabla f(\xx), \yy-\xx} \geq \frac{1}{2} \norm{\yy-\xx}_M^2 \,%
 , \quad \xx,\yy \in \R^n \, ,
\end{align}
holds. 
Let $\xx^*$ be the unique minimizer of a strongly convex function $f$ with 
parameter $M$. Then equation~\eqref{def:stronconvex} implies this 
useful relation: 
\begin{align}
\label{eq:quadraticlower_matrix}
\frac{1}{2} \norm{\xx-\xx^*}_M^2 \leq f(\xx) - f(\xx^*) %
 \leq \frac{1}{2}\|\nabla f(\xx)\|_{M^{-1}}^2 \, , %
 \quad \forall \xx \in \R^n \, .
\end{align}
The former inequality uses $\nabla f(\xx^*)=0$, and the latter one follows
from (\ref{def:stronconvex}) via
\begin{align*}
f(\xx^*) %
 &\geq f(\xx) + \lin{\nabla f(\xx), \xx^* \! -\xx} + \frac{1}{2} \norm{\xx^* \! -\xx}_M^2 \\
 &\geq f(\xx)+ \! \min_{\yy\in\R^n} \! \left( \! \lin{\nabla f(\xx), \yy-\xx} %
                        + \frac{1}{2} \norm{\yy -\xx}_M^2 \! \right) %
 \! = f(\xx) -\frac{1}{2}\|\nabla f(\xx)\|_{M^{-1}}^2
\end{align*}
by standard calculus.
\subsection{Sampling distribution}
\label{sec:normal}
Both RP algorithms from Figure~\ref{fig:algos} rely on multivariate normal distributed search directions $\uu \in \R^n$. We write  $\uu \sim \normal(\boldsymbol{\mu}, \Sigma)$ to denote that $\uu \in \R^n$ is multivariate normally distributed with mean $\boldsymbol{\mu} \in \R^n$ and covariance $\Sigma \in \Pd_n$.
As the step sizes in~\eqref{eq:rponestep} are determined by a line search, the actual scaling of $\uu$, i.e. $\norm{\uu}$, is not relevant for the behavior of the algorithm. 
We therefore restrict ourselves to \emph{normalized} search directions.
%
%
%
\begin{definition}[Normalized distribution]
\label{def:nmvn}
Let $\Sigma \in \Pd_n$. We denote by $\bar{\normal}(\0, \Sigma)$ the distribution arising from the image of the normal distribution $\normal(\0, \Sigma)$ under the mapping $T(\mathbf{x}) = \mathbf{x}/{\norm{\mathbf{x}}_{\Sigma^{-1}}}$.
\end{definition}
For example, $\uu \sim \bar{\normal}(\0, I_n)$ denotes the uniform distribution over all unit length vectors subject to the standard Euclidean norm (the uniform distribution on the unit $(n-1)$-sphere). The following lemma summarizes some facts for the normalized distribution.
\begin{lemma}
\label{lemma:factsnormalized}
Let $\vv \sim \bar{\normal}(\0, \Sigma)$ normalized with $\Sigma \in \Pd_n$ and let $A \in \Sym_n$. Then
\begin{align*}
 \E \left[\vv \vv^T \right] &= \frac{\Sigma}{n} \,, &%
 \E \left[\vv^T A \vv \right] &= \frac{\Tr[A\Sigma]}{n}\,, &%
 \E \left[(\vv^T A \vv)^2 \right] &= \frac{\Tr[A \Sigma]^2+ 2 \Tr[(A\Sigma)^2]}{n(n+2)} 
\end{align*}
and for $\xx \in \R^n$,
\begin{align*}
\expect \left[ \lin{\xx,\vv} \vv \right] = \frac{\Sigma \xx}{n} \, , %
 \quad &\text{and} \quad %
 \expect \left[ \norm{\lin{\xx,\vv} \vv}_A^2 \right] = %
         \frac{\Tr[A \Sigma] \norm{\xx}_{\Sigma}^2 %
         + 2 \norm{\xx}_{\Sigma A \Sigma}^2}{n(n+2)} \, .
\end{align*}
\end{lemma}
The proof can be found on page~\pageref{proof:lemma:normalized} in the appendix.
\subsection{Approximate line search oracles}
\label{sec:commentls}
Access to an exact line search oracle~\eqref{def:exactlinesearch} is typically not required to establish convergence of the RP algorithms. This is of importance in practical applications.
%
Commonly used line search oracles often aim at satisfying the well-known Armijo-Goldstein~\cite{Goldstein:1965,Armijo:1966}, and Wolfe~\cite{Wolfe:1969a,Wolfe:1969b} conditions. These condition measure the quality of a single search step in terms of the squared norm of the gradient. Thus, we also provide an analogous quality criterion in the full quadratic model---in addition to a slightly stronger relative accuracy measure.
\begin{definition}[Approximate line search oracles]
\label{def:rells3}
For $0 \leq \mu \leq 1$, $\xx \in \R^n$, $\uu \in \R^n \setminus \{\0\}$  and a convex function $f$, 
a function 
$\als_f \colon \reals^n \times \R^n \to \reals$ with
\begin{align}
f(\xx + \als_f(\xx,\uu)\cdot \uu) \leq f(\xx) - \mu (f(\xx) - f(\xx + \ls_f(\xx,\uu)\cdot\uu))   \tag{L1}\label{def:rells}
\end{align}
is called an approximate line search oracle with  \emph{relative accuracy} $\mu$.

For a differentiable convex function $f \in C^1_L$, 
a function 
$\als_f \colon \reals^n \times \R^n \to \reals$ with
\begin{align}
f(\xx + \als_f(\xx,\uu)\cdot \uu) \leq f(\xx) - \frac{\mu \lin{\nabla f(\xx,\uu}^2}{2 \norm{\uu}_L^2} \tag{L2}\label{def:rells2}
\end{align}
is called an approximate line search oracle with \emph{sufficient decrease} $\mu$.
\end{definition}
%
As we measure deviations only on a relative and not on an absolute scale, such an inexact line search oracle can efficiently be implemented with binary search (dichotomy), using  function evaluations only. It can easily be seen that the first condition~\eqref{def:rells} is stronger than~\eqref{def:rells2}.
\begin{lemma}[\ref{def:rells}) $\Rightarrow$ (\ref{def:rells2}]
\label{lemma:l1tol2}
Let $\xx \in \R^n$, function $f \in C^1_L$, search direction $\uu \in \R^n \setminus \{\0\}$ and $\als_f$ a line search oracle with relative accuracy $\mu > 0$. Then $\als_f$ satisfies the sufficient decrease condition~\eqref{def:rells2}.
\end{lemma}
\begin{proof}
As a simple consequence of~\eqref{def:exactlinesearch} we have 
$f(\xx + \ls_f(\xx, \uu)\cdot \uu) \leq f(\xx + t \uu)$ 
for every $t \in \R$. 
We use the quadratic upper bound~\eqref{def:lipschitz} to derive an upper bound on $f(\xx + t \uu)$. Assume $\mu = 1$. Therefore
\begin{align}
 f(\xx + \ls_f(\xx,\uu) \uu) &\leq f(\xx) + \min_{t \in \R} \left(t \lin{\nabla f(\xx), \uu} + t^2 \frac{1}{2} \norm{\uu}_L^2 \right) \label{eq:quad:continue} 
\end{align}
And the lemma follows by the (now optimal) choice $t = -\frac{\lin{\nabla f(\xx), \uu}}{2 \norm{\uu}_L^2}$. The  general case $\mu < 1$ follows straightforwardly from definition~\eqref{def:rells2}. \qed
\end{proof}
Most of our convergence results hold for both approximate line search oracles, but Theorem~\ref{thm:3} will rely on the stronger oracle~\eqref{def:rells}. We would like to remark that our convergence results to more general settings. For instance, we can vary the accuracy parameter $\mu$ in every iteration as long as $\mu$ stays positive, or the distribution of $\mu$ is independent of $\xx$ and $\uu$ (see also the concrete implementations in Section~\ref{sec:imple}).

So far, we did not discuss line search oracles with \emph{absolute} errors. We will comment on such oracles in Section~\ref{sec:lsabsolute} below.
%
%
%
\subsection{Convergence factors}
\label{sec:factors}
The following notation will be useful to formulate the convergence results form Section~\ref{sec:convergence} below.
The condition number $\kappa(A)$ of a positive definite matrix $A \in \Pd_n$ is defined as the ratio $\kappa(A):= \frac{\lambda_{\rm max}(A)}{\lambda_{\rm min}(A)}$ of the two most extreme eigenvalues. The quantities that we introduce now, can be viewed as a generalization of this concept. 
For $A,B,C,D \in \Pd_n$, and $\yy \in \R^n$ let
\begin{align}
\label{def:kappaE}
\kappaE(A,B,C,\yy) &:= \frac{ \Tr[AB] \sigma_{A,B}(\yy) + 2}{\lambda_{\rm \min}(C)(n+2)}\,, 
& \sigma_{A,B}(\yy):= \frac{\norm{\yy}_{(ABA)^{-1}}^2}{\norm{\yy}^2_{A^{-1}}} \,,
\end{align}
and
\begin{align}
\label{def:kappaT}
\kappaT(D,C) &:= \frac{\Tr[D] \lambda_{\rm min}^{-1}(D)  + 2}{\lambda_{\rm \min}(C) (n+2)}\,.
\end{align}
For brevity, we abbreviate $\kappaT(D):= \kappaT(D,I_n)$.
\begin{lemma}
\label{rem:compare}
Let $A,B,C \in \Pd_n$, and $\yy \in \R^n$. Then
\begin{align*}
0 < \kappaE(A,B,C,\yy) \leq \kappaT(AB,C) \leq \frac{\frac{1}{n}\Tr[AB] \lambda_{\rm min}^{-1}(AB)}{\lambda_{\rm \min}(C)} \leq \frac{\kappa(AB)}{\lambda_{\rm \min}(C)}\,.
\end{align*}
\end{lemma}
\begin{proof}
We show the inequalities one by one. For the first one it is enough to show that $\Tr[AB]$ is positive. Let $A^{1/2} \in \Pd_n$ denote the positive definite root of $A$. Then $AB$ and $A^{1/2}BA^{1/2} \in \Pd_n$ have the same eigenvalues, as already mentioned in Section~\ref{sec:quadnorms} above, see e.g.~\cite[Prop.~13.2]{Puntanen:2011}. 
For the second one we use Lemma~\ref{lem:AvsBnorm} to find a uniform upper bound on $\sigma_{A,B}(\yy)$:
\begin{align}
\label{eq:error2}
\sigma_{A,B}(\yy) = \frac{\norm{\yy}_{(ABA)^{-1}}^2}{\norm{\yy}^2_{A^{-1}}} \leq \lambda_{\rm max}(A^{-1} B^{-1}) = \lambda_{\rm min}^{-1}(AB)\,.
\end{align}
For $a\geq b > 0$ it holds $\frac{a+c}{b+c} \leq \frac{a}{b}$ for any $c \geq 0$. Therefore, the choice $a = \Tr[AB] \lambda_{\rm min}^{-1}(AB)$, $b =n$ and $c= 2$ implies the third inequality. The last one is trivial.
\qed
\end{proof}
\section{Convergence of Fixed Metric Random Pursuit}
\label{sec:convergence}

We will now derive the global convergence rates for Algorithm~F-RP on convex and strongly convex functions. To prepare the proof, we first study the expected progress in a single step,
which is the quantity

\[
f(\xx_k) - \E \left[ f(\xx_{k+1}) \mid \xx_k \right] \, .
\]
For the first two theorems, it suffices to assume access to an approximate line search oracle with~\eqref{def:rells2}. However, for Theorem~\ref{thm:3}, the stronger~\eqref{def:rells} is required.
\subsection{Single step progress}
Once a search direction is determined, the subsequent iterate is chosen according to~\eqref{eq:rponestep}. As the step size is determined by the line search oracle, we can derive the following lower bound on the single step progress.
\label{sec:singlestep}
\begin{lemma}[Single step progress of~\eqref{def:rells2}]
\label{lemma:new_onestep}
Let $f \in C_{L}^1$, $\xx \in \R^n$ such that $\nabla f(\xx) \neq \0$, covariance $\Sigma \in \Pd_n$ direction $\uu \sim \bar{\normal}(\0, \Sigma)$, and $\als$ an approximate line search oracle~\eqref{def:rells2} with sufficient decrease $0 \leq \mu \leq 1$ and let
$\xx_+ = \xx + \als_f(\xx, \uu)\cdot \uu$ the next iterate after one step of Algorithm F-RP. Then
\begin{align*}
\E_\uu \left[ f(\xx_+) \mid \xx \right] 
&= f(\xx) - \frac{\mu}{2 n \kappaE(L,\Sigma,I_n, \nabla f(\xx))} \norm{ \nabla f(\xx)}^2_{L^{-1}} \\ &\leq f(\xx) - \frac{\mu}{2 n \kappaT(L\Sigma, I_n)} \norm{ \nabla f(\xx)}^2_{L^{-1}} \,.
\end{align*}
where 
$\kappaE$ and $\kappaT$ 
as in Section~\ref{sec:factors}.
\end{lemma}
\begin{proof}
All we need to find is a lower bound on the conditional expectation
\begin{align}
E_L := \E \left[\frac{\lin{\nabla f(\xx),\uu}^2}{2 \norm{\uu}_L^2} \mid \xx \right]\,, \label{eq:el}
\end{align}
of the expression on the right hand side of~\eqref{def:rells2}. Expressions for such expected values have been derived in the literature (see e.g.~\cite{Mathai:1992}), but no simple closed form solutions exist. As we here only need a lower bound, we can apply the following trick. For fixed $\yy \in \R^n$ and $\uu \in \R^n \setminus \{0\}$ we observe
\begin{align*}
 \frac{\lin{\yy, \uu}^2}{\norm{\uu}_L^2} &= \max_{t} \left(2t \lin{\yy, \uu} - t^2 \norm{\uu}_L^2 \right) \\ 
 &\geq \max_h \left(2 h \lin{(L\Sigma)^{-1} \yy, \uu} \lin{\yy, \uu} - h^2 \norm{ \lin{ (L \Sigma)^{-1} \yy, \uu} \uu}_L^2 \right) \,,
\end{align*}
where the equality follows by standard calculus, and the inequality by suboptimally setting $t = h \lin{(L\Sigma )^{-1}\yy, \uu}$. With Lemma~\ref{lemma:factsnormalized} we can compute the expectation of the terms inside the maximum. We have
\begin{align*}
\E_\uu \bigl[\lin{(L \Sigma)^{-1} \yy, \uu} \lin{\yy, \uu} \mid \xx\bigr] &= \E_\uu \bigl[ \uu^T (L \Sigma)^{-1} \yy \yy^T \uu \mid \xx \bigr] 
= \frac{1}{n} \norm{\yy}_{L^{-1}}^2\,, 
\end{align*}
and
\begin{align*}
\E_\uu \left[  \norm{ \lin{ (L\Sigma)^{-1} \yy, \uu} \uu}_L^2 \mid \xx \right] &= \frac{\Tr[L \Sigma] \norm{\yy}_{(L \Sigma L)^{-1}}^2  + 2 \norm{\yy}^2_{L^{-1}}}{ n (n+2) }\,.
\end{align*}
By Jensen's inequality it is indeed valid to interchange the expectation with the maximum. We have
\begin{align*}
E_L &\geq \max_h \left(2h \frac{\norm{\yy}_{L^{-1}}^2}{n} - h^2 \frac{\Tr[L \Sigma] \norm{\yy}_{(L\Sigma L)^{-1}}^2 + 2 \norm{\yy}_{L^{-1}}^2}{n (n+2)} \right) \\
&\geq \frac{ (n+2) \norm{\yy}_{L^{-1}}^4}{n \bigl(\Tr[L \Sigma] \norm{\yy}_{(L \Sigma L)^{-1}}^2 + 2 \norm{\yy}_{L^{-1}}^2\bigr)} ,,
\end{align*}
where $h$ was chosen to maximize the expression in the bracket, i.e.
\begin{align*}
- (n+2) \norm{\yy)}^2_{L^{-1}} + h \left(\Tr[L \Sigma] \norm{\yy}_{(L\Sigma L)^{-1}}^2  + 2 \norm{\yy}^2_{L^{-1}}\right) = 0\,.
\end{align*}
This choice of $h$ implies the first inequality for $\yy = \nabla f(\xx)$. The second one follows directly from Lemma~\ref{rem:compare}. \qed
\end{proof}

The line search oracle with absolute accuracy~\eqref{def:rells} achieves a single step progress that is as least as good as the bound derived in Lemma~\ref{lemma:new_onestep} above. However, we are more flexible and an can also derive a bound that does not scale directly with $\norm{\nabla f(\xx)}_{L^{-1}}^2$. 
\begin{lemma}[Single step progress of~(\eqref{def:rells}]
\label{lemma:new_onestep2}
Let $f \in C_{L}^1$, $\xx \in \R^n$, covariance $\Sigma \in \Pd_n$ direction $\uu \sim \bar{\normal}(\0, \Sigma)$, and $\als$ an approximate line search oracle~\eqref{def:rells} with relative accuracy $0 \leq \mu \leq 1$ and let
$\xx_+ = \xx + \als_f(\xx, \uu)\cdot \uu$ the next iterate after one step of Algorithm F-RP. In addition, let $\xx^* \in \R^n$ be one of the minimizers of $f$.  Then for every positive $h\geq0$ it holds
\begin{align*}
\E_\uu \left[ f(\xx_+) - f(\xx^*) \mid \xx \right] \leq \! \left( 1-\frac{h \mu}{ n } \right) \left( f(\xx) - f(\xx^*) \right) + \frac{h^2 \mu \kappaT(L\Sigma)}{2 n} \norm{\xx-\xx^*}_{L}^2,,
\end{align*}
where $\kappaT$ as in Section~\ref{sec:factors}.
\end{lemma}
\begin{proof}
As in the proof of Lemma~\ref{lemma:l1tol2} we use a supobtimal choice of the unknown optimal value $\ls_f(\xx,\uu)$ together with the quadratic upper bound. Here we use in~\eqref{eq:quad:continue} the value $t= h\lin{\Sigma^{-1} (\xx-\xx^*), \uu}$. This leads to
\begin{align*}
f(\xx_+) \leq f(\xx) - h\mu \lin{\Sigma^{-1} (\xx-\xx^*), \uu}\lin{\nabla f(\xx), \uu} + \frac{h^2 \mu}{2} \norm{\lin{\Sigma^{-1} (\xx-\xx^*), \uu} \cdot \uu}_L^2\,.
\end{align*}
With Lemma~\ref{lemma:factsnormalized} we can again compute the conditional expectation of the terms on the right hand side:
\begin{align*}
\E_\uu [\lin{\Sigma^{-1} (\xx - \xx^*), \uu} \uu \mid \xx] &= \frac{1}{n}\left(\xx-\xx^*\right)\,, \\%
\E_\uu \left[  \norm{\lin{\Sigma^{-1} (\xx - \xx^*), \uu} \cdot \uu}_{L}^2 \mid \xx \right] &= \frac{\Tr[L \Sigma] \norm{\xx-\xx^*}_{\Sigma^{-1}}^2  + 2 \norm{\xx - \xx^*}^2_{L}}{ n (n+2) }\,,
\end{align*}
and obtain
\begin{multline}
\label{eq:onestep}
\E_\uu[ f(\xx_+) \mid \xx] \leq f(\xx) - \frac{h \mu}{ n} \lin{\nabla f(\xx), \xx-\xx^*} \\ + \frac{h^2 \mu }{2 n (n+2)} \left(\Tr[L \Sigma] \norm{\xx-\xx^*}_{\Sigma^{-1}}^2  + 2 \norm{\xx-\xx^*}^2_{L}\right)\,.
\end{multline}
Using the definition of convexity
(see the beginning of Section~\ref{sec:qbounds}) we can bound the term $\lin{\nabla f(\xx), \xx-\xx^*}$ from below by $f(\xx) -f(\xx^*)$.  Finally, we bound the $\Sigma^{-1}$-norm from above with Lemma~\ref{lem:AvsBnorm}. As in the proof of Lemma~\ref{rem:compare} we get
\begin{align*}
\norm{\xx-\xx^*}_{\Sigma^{-1}}^2 &\leq \lambda_{\rm
  max}(L^{-1} \Sigma^{-1}) \norm{\xx-\xx^*}_{L}\,, 
\end{align*}
and the lemma follows from $\lambda_{\rm
  max}(L^{-1} \Sigma^{-1}) = 1/\lambda_{\rm min}(\Sigma L)$. \qed 
\end{proof}

The previous two lemmas shows that, on average, there is progress in every single
step if either $\norm{\nabla f(\xx)}_{L^{-1}}$ or $\norm{\xx-\xx^*}_L^2$ is bounded away from
zero.\footnote{Here we use that $\Tr \left[L \Sigma\right]> 0$; see
  the proof of Lemma~\ref{rem:compare} in Section~\ref{sec:factors}.} 
  This leads us to the next section where we will use the just derived lemmas to prove global convergence.
%
%
\subsection{Global convergence}
We now use the previously derived bounds on the expected
single step progress (Lemma~\ref{lemma:new_onestep} and~\ref{lemma:new_onestep2}) to show convergence of F-RP
in expectation. We first show convergence on smooth but not
necessarily strongly convex functions.
\begin{theorem}    
\label{thm:convex}
Let $f \in C^1_{L}$, let $\xx^* \in \R^n$ be a minimizer of $f$ and let the sequence $\{\xx_k\}_{k \geq 0}$ be generated by Algorithm F-RP with covariance $\Sigma \in \Pd_n$ and line search~\eqref{def:rells2} with sufficient decrease $0< \mu \leq 1$.
Assume there exists $R \in \R$, s.t. $\norm{\yy - \xx_0}_{L} \leq R$ for all $\yy \in \R^n$ with $f(\yy) \leq f(\xx_0)$. Then, for any $N \geq 0$, we have
\begin{align*}
\E \left[ f(\xx_{N}) - f(\xx^*)\right] \leq \frac{Q}{N+1} \,,
\end{align*}
where 
\begin{align*}
Q &:= \max \left\{\frac{2 n R^2 \kappaT(L\Sigma)}{\mu}, f(\xx_0) - f(\xx^*) \right\} \,. & %
\end{align*}
\end{theorem}
\begin{proof}
We will apply the bound on the single step progress from Lemma~\ref{lemma:new_onestep}, but first, let us derive a lower bound on the norm $\norm{\nabla f(\xx)}_{L^{-1}}^2$ for $\xx \in \R^n$ with $\norm{\xx - \xx^*}_L \leq R$. By convexity~\eqref{def:convex1} and the assumptions on $\xx$ 
we have 
\begin{align*}
 f(\xx) - f(\xx^*) \leq  \lin{\nabla f(\xx), \xx^*-\xx} \leq R \norm{\nabla f(\xx)}_{L^{-1}} \,.
\end{align*}
Hence, by Lemma~\ref{lemma:new_onestep} we can estimate the single step progress
\begin{align*}
 \E \left[f(\xx_+) \mid \xx \right] \leq f(\xx) - \tau \bigl( f(\xx) - f(\xx^*) \bigr)^2\,,
\end{align*}
where $\tau:= \frac{\mu}{2 n R^2 \kappaT(L\Sigma)}$ and $\xx_+ = \xx + \als_f(\xx,\uu)\cdot \uu$ with the notation from Lemma~\ref{lemma:new_onestep}. Conditioned on $\xx$, the quantity $f(\xx)- f(\xx^*) =: f_\xx$ is just a constant. Hence, subtracting $f(\xx^*)$ on both sides, we can rewrite this bound as
\begin{align}
\E \left[f(\xx_+) - f(\xx^*) \mid \xx \right]  &\leq  f_\xx - \tau f_\xx^2 = f_\xx + 2 \min_{h} \left(-h f_\xx + \frac{h^2}{2 \tau} \right) \notag \\ &\leq (1-2h) f_\xx +h^2\tau^{-1} \label{eq:recurse} \,,
\end{align}
where the last inequality holds for arbitrary parameter $h \in \R$.

Now we can proceed to analyze the multi step behavior. For this, we just repeatedly apply the bound~\eqref{eq:recurse} on the single step progress. Conditioning on $\{\xx_k\}_{k= 0}^{N-1}$, we estimate
\begin{align*}
 \E \left[ f(\xx_N) - f(\xx^*) \mid \{\xx_k\}_{k= 0}^{N-1} \right] \leq (1-2 h_N ) f_\xx + \frac{h_N^2}{\tau}\,,
\end{align*}
for any parameter $h_N \in \R$. Now, formally, we recursively apply the conditional expectations,  onditioning on $\{\xx_k\}_{k= 0}^{N-2}$, $\{\xx_k\}_{k= 0}^{N-3}$, \dots , $\{\xx_0\}$, and use~\eqref{eq:recurse} with different parameters $h_{N-1},\dots,h_1$ in every step. By the tower property of conditional expectations, we end up with a bound on $\E[f(\xx_N)- f(\xx^*)]$ that depends on the free parameters $h_1,\dots h_{N}$. 
As in \cite[Theorem~5.3]{Stich:2011}, 
the choice $h_k := \frac{1}{k}$ for $k=1, \dots, N$ yields the lemma (see also~\cite[Lemma~A.1]{Stich:2011}). \qed

\end{proof}
On strongly convex functions the convergence of F-RP is linear.
\begin{theorem}
\label{thm:global_strong}
Let $f \in C^1_{L}$ and let $f$ in addition be strongly convex with parameter $M \in \Pd_n$. Let $\xx^* \in \R^n$ denote the unique minimizer of $f$,  and let the sequence $\{\xx_k\}_{k \geq 0}$ be generated by Algorithm F-RP with covariance $\Sigma \in \Pd_n$ and line search with accuracy $0 \leq \mu \leq 1$. Then
\begin{align*}
\E \left[ f(\xx_{N})-f(\xx^*) \right] %
  &\leq \left(1 - \frac{\mu}{n \kappaT(L\Sigma,M)} \right)^N  \cdot  \left(f(\xx_0)-f(\xx^*) \right) \, .
\end{align*}
\end{theorem}
\begin{proof}
We use Lemma~\ref{lem:AvsBnorm} to establish
\[
\norm{\nabla f(\xx_k)}_{L^{-1}}^2 \geq \lambda_{\rm \min}(M L^{-1}) \norm{\nabla f(\xx_k)}_{M^{-1}}^2 \,.
\]
Applying the quadratic lower bound~\eqref{eq:quadraticlower_matrix} to further bound the latter term from below yields
\[
\norm{\nabla f(\xx_k)}_{L^{-1}}^2 \geq 2 \lambda_{\rm \min}(M L^{-1}) \left(f(\xx_k) - f(\xx^*)\right)\,.
\]
Now we can combine this bound with Lemma~\ref{lemma:new_onestep} and get
\begin{align}
\label{eq:oneiter}
\E_\uu \left[ f(\xx_k) - f(\xx^*) \mid \xx_k \right] \leq \varrho(\xx_k) \cdot \left( f(\xx_k)- f(\xx^*) \right)\,,
\end{align}
where
\begin{align}
\label{def:confactor}
\varrho(\xx_k) := 1-  \frac{\mu }{n \kappaE(L,\Sigma,M,\nabla f(\xx_k)) }
 \leq 1 - \frac{\mu}{n \kappaT(L\Sigma,M)} \,,
\end{align}
is the exact convergence factor. The uniform upper bound was established in Lemma~\ref{rem:compare}.
The Theorem follows now 
by taking expectation over $\xx_k$. \qed
\end{proof}
We remark that the progress is strict: by Lemma~\ref{rem:compare} the \emph{convergence factor}
\begin{align}
\label{def:confactor2}
\hat{\varrho} := 1 - \frac{\mu}{n \kappaT(L\Sigma,M)} \,,
\end{align}
is strictly smaller than one.


%

It is not necessary that the function $f$ is strongly convex everywhere for linear convergence to hold. Theorem~\ref{thm:3} below shows that convergence (at about a quarter of the rate of the one in Theorem~\ref{thm:global_strong}) can be proven assuming only a weaker condition. Let us recall that strong convexity with parameter $M$ implies that 
\begin{align}
\label{eq:relaxstrongconvex}
f(\xx) - f(\xx^*) \geq \frac{1}{2} \norm{\xx - \xx^*}_M^2\,, \forall \xx \in \R^n \,.
\end{align}
It turns out that, instead of strong convexity~\eqref{def:stronconvex}, the weaker condition~\eqref{eq:relaxstrongconvex} is enough for linear convergence. Strong convex functions need to have positive curvature everywhere, whereas functions with~\eqref{eq:relaxstrongconvex} could also be linear on bounded subsets. 
%
\begin{theorem}
\label{thm:3}
Let $f \in C^1_{L}$ and let $f$ in addition have a unique minimizer $\xx^* \in \R^n$ satisfying~\eqref{eq:relaxstrongconvex} with $M \in \Pd_n$. Let the sequence $\{\xx_k\}_{k \geq 0}$ be generated by Algorithm F-RP with covariance $\Sigma \in \Pd_n$ and line search oracle~\eqref{def:rells} with relative accuracy $0 \leq \mu \leq 1$. Then
\begin{align}
 \E \left[ f(\xx_{N})-f(\xx^*) \right]  \leq \left(1-\frac{\mu}{4 n \kappaT(L\Sigma,ML^{-1})}\right)^N \cdot \left(f(\xx_0)-f(\xx^*) \right)
\end{align}
\end{theorem}
\begin{proof}
First, we use Lemma~\ref{lem:AvsBnorm} followed by~\eqref{eq:relaxstrongconvex} to estimate 
\begin{align*}
\kappaT(L\Sigma) \norm{\xx-\xx^*}_{L}^2 &\leq 
\kappaT(L\Sigma,ML^{-1}) \norm{\xx-\xx^*}_{M}^2 \\ &\leq 2 \kappaT(L\Sigma,ML^{-1}) (f(\xx) - f(\xx^*))\,.
\end{align*}
Now we can just apply Lemma~\ref{lemma:new_onestep2} to estimate the single step progress as
\begin{align}
\E \left[ f(\xx_+) - f(\xx^*) \mid \xx \right] 
\notag \,,
&\leq \bigg(1- \frac{h \mu}{n}  + \frac{h^2 \mu \kappaT(L\Sigma,ML^{-1})}{n}  \bigg) \left( f(\xx) - f(\xx^*) \right) 
\notag \,.
\end{align}
%
By setting $h^{-1} = 2 \kappaT(L\Sigma,ML^{-1})$, the term in the left bracket 
becomes $\left(1-\frac{\mu}{4 n \kappaT(L\Sigma,ML^{-1})}\right)$ and the proof continues as the proof of Theorem~\ref{thm:global_strong}. \qed
\end{proof}
\subsection{Discussion of the Results}
\label{sec:lsabsolute}
The presented theoretical results extend our previous work in~\cite{Stich:2011} in two ways: (i) the analysis in~\cite{Stich:2011} considered only F-RP with covariance $\Sigma = I_n$ the $n$-dimensional identity matrix with less expressive quadratic lower and upper bound assumptions;   (ii) the lower- and upper bounds introduced in Section~\ref{sec:qbounds} allow for a more detailed description of the convergence rates because the quadratic model captures the eigenspectra of the functions.

We see in Theorem~\ref{thm:global_strong} that the number of iterations of F-RP algorithm to reach a target accuracy is proportional to $\mu^{-1}$. This means that for instance for $\mu=\frac{1}{2}$, only twice as many iterations are necessary to reach the same accuracy as with the choice $\mu = 1$, respectively.

The results from the Theorems~\ref{thm:convex}--\ref{thm:3} can also be extended to accommodate for more general line search oracles, for instance also with 
\emph{additive error} of a constant $\epsilon > 0$ in every step.
Such errors are not crucial, the additional error terms just have to be carried along. We refer the interested reader to~\cite{Stich:2011}, where such analysis has been carried out for a similar problem. For functions that admit linear convergence (i.e. Theorem~\ref{thm:global_strong} and~\ref{thm:3}), these errors add up to an absolute constant $C(\epsilon) = \Theta(\epsilon)$ that does not depend on the number $N$ of iterations. On convex functions as treated in Theorem~\ref{thm:convex}, the error grows as $\epsilon N$ with the number of iterations, leading to divergence if the number $N$ of iterations is too large. Therefore we see, that it is much better to express the errors in terms of the relative parameter $\mu$ instead of absolute values.

All results can also be generalized to the case when the accuracy (the relative $\mu$ and possible additive $\epsilon$) of the line search oracles changes in every iteration. This amounts to different bounds on the single step progress in every iteration, and the summation in the proofs of Theorems~\ref{thm:convex}--\ref{thm:3} becomes slightly more involved (see e.g.~\cite{Stich:2014thesis}).

As a last extension, we would like to point out that the results can also be generalized to different sampling distributions and are not only valid for $\bar{\normal}(\0,\Sigma)$ vectors. However, the actual bounds on the convergence rate may change, depending on the new distribution. To determine the new convergence factors, one only has to calculate the expectation $E_L$ in~\eqref{eq:el} for the new distribution, the rest of the proof remains the same.

\subsection{Illustration of the results}
\label{sec:numericsec4}
Let us illustrate the derived bounds with an example. For simplicity, we consider a quadratic function $f(\xx)= \frac{1}{2} \norm{\xx}_A^2$. Clearly, $f \in C^1_A$ and $f$ is strongly convex with parameter $A$. The algorithm F-RP with covariance $\Sigma = I_n$ converges on this function according to Theorem~\ref{thm:global_strong}, and the convergence rate is described by the convergence factor $\hat{\varrho}$ from~\eqref{def:confactor2}. For exact line search ($\mu = 1$) we have $\hat{\varrho} = \bigl(1 - \frac{1}{n \kappaT(A,A)}\bigr)\leq \bigl(1 - \frac{\lambda_{\rm min}(A)}{\Tr[A]}\bigr)$, where the last estimate follows from Lemma~\ref{rem:compare}. We see that this is an improvement over the factor $\bigl(1- \frac{1}{n \kappa(A)}\bigr)$ derived in~\cite{Stich:2011} if the average of the eigenvalues of $A$ is much smaller than the maximal one.

To demonstrate this, let us consider a class of quadratic functions, $g_i \colon \R^n \to \R$ for $1 \leq i < n$, with parameter $\ell \geq 1$:
\begin{align*}
g_i(\xx) &= \frac{\ell}{2} \sum_{j=1}^{i} x_i^2 + \frac{1}{2} \sum_{j= i+1}^{n} x_i^2 
\end{align*}
The Hessians of all functions $g_i$ have the same maximal ($\ell$) and minimal ($1$) eigenvalues. The functions have two different scales that are distributed among the dimensions according to the parameter $i$. 
A previous numerical study~\cite{Stich:2012a} suggests that function $g_i$ is challenging for RP algorithms if $i$ is large (here we use $g_{\lceil\frac{n}{2}\rceil}$ as in~\cite{Stich:2012a}), and easy for $i$ small (here we use $g_5$).  
Figure~\ref{fig:sec4} shows the numerically observed convergence rates (black lines) of F-RP with exact line search for functions $g_{25}$ and $g_5$ with $\ell=1000$ in dimension $n=50$.

The algorithm F-RP with $\Sigma = I_n$ converges on both functions where the convergence rate can be estimated by the convergence factor (cf. Theorem~\ref{thm:global_strong}). For both functions, $g_{25}$ and $g_5$, the result established in~\cite{Stich:2011} provides the same upper bound $\left(1-\frac{1}{\ell n}\right)$ on the convergence factor. Our new result provided in Theorem~\ref{thm:global_strong} yields two different estimates for these two functions, namely $\bigl(1- \frac{2}{\ell n} \bigr)$ for $g_{25}$ and roughly $\bigl(1 - \frac{1}{5 \ell}  \bigr)$ for $g_5$  (see Table~\ref{tab:sec4}). This is in agreement to the empirical observations, as F-RP converges faster on $g_5$ than on $g_{25}$. However, the presented bounds (red lines) slightly underestimate the rate.
\begin{table}
\begin{center}
\begin{tabular}{ c | c | c }
  \hline  
   Function  & Previous result~\cite{Stich:2011} & New result \\  \hline
   \multirow{2}{*}{$g_i$}          & $1 - \frac{1}{n \kappa(A)}$ & $1 - \frac{\lambda_{\rm min}(A)}{\Tr[A]}$  \\
    \cline{2-3} 
   & $1-\frac{1}{\ell n}$ & $ 1- \frac{1}{i \ell + (n-i)}$ \\
  \hline  
\end{tabular}
\caption{Theoretical convergence rates of F-RP on functions $g_i$ from previous analysis in~\cite{Stich:2011}, and Theorem~\ref{thm:global_strong}.}
\label{tab:sec4}
\end{center}
\end{table}
\begin{figure}
\centering
\includegraphics[width=.99\textwidth]{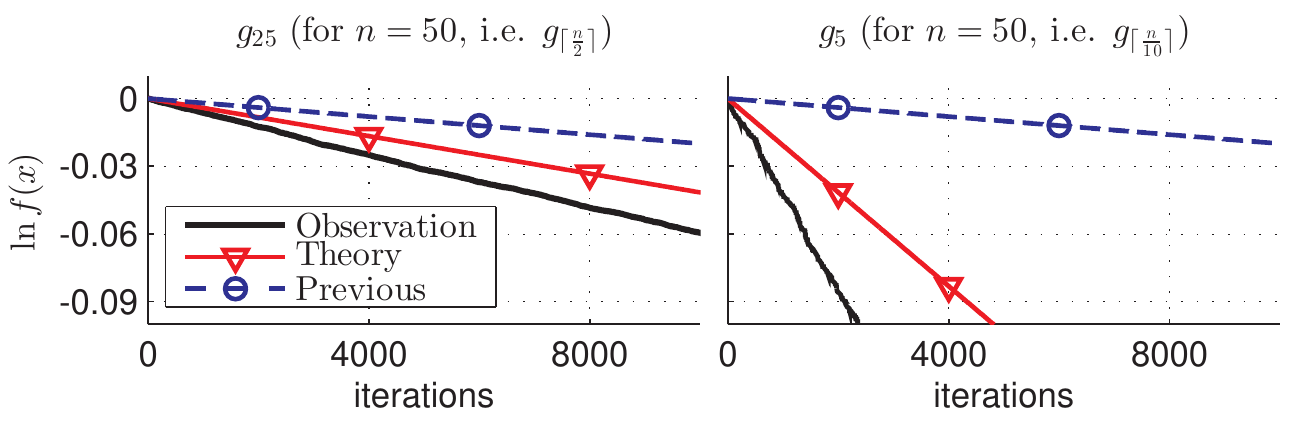}
\caption{Comparison of the derived convergence rates (see Table~\ref{tab:sec4}) with empirical measurements on $g_{25}$ (left) and $g_5$ (right) with $\ell=1000$ in dimension $n=50$. Logarithm of function value vs. number of iterations (ITS). Observed F-RP convergence (solid/black), convergence rate derived in \cite{Stich:2011} (dashed), convergence rate derived in this paper (solid/red).}
\label{fig:sec4}
\end{figure}

It is clear that our worst-case analysis cannot give accurate convergence rates on all convex functions. We can, nonetheless, give a pointer to the part of the proof of Theorem~\ref{thm:global_strong} where we clearly use too conservative estimates. In Equation~\eqref{def:confactor} we
 used a crude estimation of the factor $\varrho(\xx_k)$. This estimate is not tight in every step $k$ (but in the worst case), as can easily seen from Equation~\eqref{eq:error2} in the proof of Lemma~\ref{rem:compare}.  In order to find a convergence factor that best matches the observed rates we may rather analyze an \emph{average case} scenario, and consider the expected value of $\sigma(\xx_k)$ over the trajectory $\{\xx_k\}_{k \geq 0}$ (see Lemma~\ref{remark:jensenforquad} below). However, it seems that such an analysis is the scope of this manuscript as we would not only need the expected function values $\E[f(\xx_k)])$, but also precise information on $\xx_k$ itself. 
Intuitively, we would expect the iterates to be almost always at the ``far ends'' of the ellipsoidal level sets $\{\xx\colon \xx^TA\xx=c\}$ of $f$, and therefore~\eqref{eq:error2} might be only be improved by a small factor.
\begin{lemma}
\label{remark:jensenforquad}
Let $\sigma \colon \R^n \to \R_{\geq 0}$ nonnegative, let $a, b > 0$, let $\{\xx_k\}_{k \geq 0}^N$ and arbitrary sequence of $N$ points in $\R^n$ and $\bar{\sigma} = \frac{1}{N} \sum_{i=1}^{N} \sigma(\xx_k)$. Then
\begin{align*}
\Pi_N:=\prod_{k=1}^N \left(1- \frac{a}{\sigma(\xx_k) + b} \right) \leq  \exp \left[ - \frac{aN}{\bar{\sigma} + b} \right]\,.
\end{align*} 
\end{lemma} 
\begin{proof}
The function $\frac{1}{x}$ for $x > 0$ is convex, therefore by Jensen's inequality 
and $1-x \leq e^{-x}$ we find
\begin{align*}
 -\ln \Pi_N \geq \sum_{k=1}^{N} \frac{a}{\sigma(\xx_k) + b} \geq \frac{aN}{\frac{1}{N} \sum_{i=0}^{N-1} \sigma(\xx_k) + b} = \frac{aN}{\bar{\sigma} + b}\,. \tag*{\qed}
\end{align*}
\end{proof}
\section{Metric Learning in Random Pursuit}
\label{sec:metriclearning}
In Section~\ref{sec:convergence} we have derived exact bounds on the progress rate of F-RP that depend on the sampling distribution. For a quadratic function $f(\xx) = \frac{1}{2} \norm{\xx}_H^2$ with $H \in\Pd_n$, the expected running time of F-RP with search directions $\uu \sim \bar{\normal}(\0,I_n)$ is $O(\kappaT(H) n \ln \frac{1}{\epsilon})$, where $\epsilon > 0$ is the desired accuracy. In contrast, for $\uu \sim \bar{\normal}(\0,H^{-1})$ the running time drops to $O(n \frac{1}{\epsilon})$. This rate is (i) independent of the function $f$ (i.e., the spectrum of $H$) and (ii)
optimal from a theoretical point of view. This follows from the fact that $\left(1- \frac{1}{n}\right)$ is a lower bound on the convergence rate of the ``hit-and-run'' search algorithm as shown by J\"{a}gersk\"{u}pper in~\cite{jaeg:07}. The idealized ``hit-and-run'' scheme analyzed there is identical to the Random Pursuit.

Computing an approximation $\hat{H}$ to $H$ with $\kappaT(\hat{H}^{-1}H) \approx 1$ and then sampling from $\bar{\normal}(\0,\hat{H}^{-1})$ instead for the optimization phase, can reduce the running time to $O(T + n \ln \frac{1}{\epsilon})$, where $T$ denotes the running time of the Hessian Estimation scheme. This Variable Metric Random Pursuit algorithm (V-RP) improves over F-RP if $T \leq \kappaT(H)\ln\frac{1}{\epsilon}$.
This approach also works for general strongly convex functions $f \colon \R^n \to \R$ where the Hessian $\nabla^2 f(\xx)$ is not necessarily constant for all $\xx \in \R^n$. If we assume that the Hessian is only mildly changing (see for instance Lemma~\ref{lemma:factor1} below) then it might suffice to find an approximation of the Hessian $\nabla^2 f(\xx_0)$ of the initial search point $\xx_0 \in \R^n$. Otherwise, we should use a scheme that can iteratively update its estimation of the Hessian, allowing for unforeseen changes. 

We are now left with the challenge of how to efficiently estimate a Hessian matrix $H$ in the present gradient-free setting. Iterative
stochastic covariance matrix adaptation schemes are well-established
in gradient-free continuous optimization
\cite{Hansen:2001,Kjellstrom:1981,Mueller:2010a} and try to estimate directly $\Sigma = H^{-1}$, but are notoriously
difficult to study theoretically. A welcome alternative has recently
been introduced by Leventhal and Lewis~\cite{Leventhal:2011} in form
of RHE. We here review and extend their scheme. For a quadratic function $f(\xx) = \frac{1}{2} \xx^T H \xx$ and initial iterate $B_0 \in \Pd_n$, Leventhal and Lewis
already showed that RHE generates a random sequence $\{B_k\}_{k \geq 0}$ of Hessian approximations with
\begin{align}
\label{eq:resultlev}
 \E[ \norm{B_k - H}_F^2 ] \leq \left(1-\frac{2}{n(n+2)}\right)^k \norm{B_0 - H}_F^2\,.
\end{align}
Therefore, if we use RHE to generate an approximation $\hat{H}$ of the Hessian and then use F-RP with sampling distribution $\bar{\normal}(\0,\hat{H}^{-1})$, the running time of this two-stage V-RP algorithm is $O(n^2 \ln \norm{B_0 - H}_F + n \ln \frac{1}{\epsilon})$ on a quadratic function.

Our contributions are twofold: on the theoretical side, we provide new insights into RHE. We show that (i) RHE itself can be viewed as an instance of F-RP and (ii) give exact expressions for the expectation $\E[ \norm{B_k - H}_F^2 ]$. Furthermore, (iii) we estimate the impact on the running time of the aforementioned two-stage V-RP algorithm if RHE converges to $\nabla^2 f(\xx_0)$, but this matrix is not a very good approximation of the Hessian at the optimum $\xx^*$, $\nabla^2 f(\xx_0) \neq \nabla^2 f(\xx^*)$. On the practical side, we (iv) present three novel and theoretically sound implementations of RHE. For many practical situations, function evaluations are the most costly operations, and the goal is to keep the number of evaluations as low as possible. The third proposed scheme allows---at the expense of $O(n^2)$ storage---to significantly boost the performance of the RHE update in this scenario. 
%

\subsection{Variable Metric update scheme}
\label{sec:lev}
RHE from Leventhal and Lewis~\cite{Leventhal:2011} comprises direct updates of a Hessian estimate. Given a symmetric matrix $B \in\Pd_n$ as current Hessian estimate, the next iterate $B_+$ is determined according to:
\begin{align}
\label{eq:leventhalrule}
B_+ = B + \uu^T \left( H - B \right) \uu \cdot \uu \uu^T \,,
\end{align}
where $\uu \sim \bar{\normal}(\0, I_n)$.
Let us now present a novel interpretation of this update which reveals that RHE is just a special instance of a F-RP algorithm. Here, the search space of the underlying optimization problem is not $\R^n$ as in Section~\ref{sec:convergence}, but $\Sym_n$, the space of symmetric matrices. 
As objective, we aim at minimizing the distance to the Hessian $H$, measured in the Frobenius norm:
\begin{align}
 g(X) := \norm{X - H}_F^2\,. \label{def:g}
\end{align}
This defines a quadratic function $g \colon \Sym_n \to \R$, and for a $B \in \Sym_n$ and a fixed `search direction' $\uu \uu^T$, we can easily derive an analytic expression of $\ls_g(B,\uu \uu^T)$, the exact line search in direction $\uu \uu^T$. By definition, we have
\begin{align*}
 \ls_g(X,\uu \uu^T) = \argmin_t g(B + t \uu \uu^T) = \argmin_t \norm{B + t \uu  \uu^T -H }_F^2\,.
\end{align*}
We now determine the parameter $t$ as to minimize the right hand side, that is 
$ \uu^T B \uu -\uu^T H \uu + t \uu \uu^T \uu \uu^T  = 0$
and conclude
\begin{align}
 \ls_g(B,\uu \uu^T) = \uu^T (H - B) \uu \,. \label{eq:lsopt}
\end{align}
We have now established, that the RHE is (i) just an instance of a specific F-RP algorithm. Moreover, (ii) the update step in~\eqref{eq:leventhalrule} corresponds to a step of F-RP with an exact line search oracle $\ls_g(B, \uu\uu^T)$. 

The formula~\eqref{eq:leventhalrule}, or equivalently~\eqref{eq:lsopt},   requires the evaluation of $\uu^T H \uu$ with unknown $H$. For twice differentiable functions $f$ the second derivative of $f$ at $\xx$ in direction $\uu$ can be well approximated by finite differences\footnote{The two points $\xx \pm \epsilon \uu$ are not required to be at the same distance to $\xx$. For points $\xx - \epsilon_1 \uu$, 
$\xx + \epsilon_2 \uu$ the curvature can be estimated by 
quadratic interpolation with slight adaptation of formula~\eqref{eq:finitediff}.}:
\begin{align}
\label{eq:finitediff}
\uu^T H \uu \approx \frac{f(\xx + \epsilon \uu) - 2 f(\xx) + f(\xx - \epsilon \uu) }{\epsilon^2} 
\end{align}
for some small $\epsilon > 0$ as proposed in~\cite{Leventhal:2011}.
In the convex quadratic case, the above formula is exact for arbitrary $\epsilon > 0$. For general functions the approximation of $\uu^T H \uu$ with formula~\eqref{eq:finitediff} may not be accurate, thus leading to a failure of the update. Note that this approach 
only requires two additional function evaluations at $\xx \pm \epsilon \uu$. In addition, the formula implies that the estimate $B_+$ behaves at $\xx$ like the unknown Hessian along direction $\uu$, that is, $\uu^T B_+ \uu = \uu^T H \uu$. This can be seen directly from~\eqref{eq:leventhalrule} by noting that $\uu^T \uu \uu^T \uu = 1$.

The interpretation of RHE as a F-RP algorithm allows to study inexact line search oracles, which correspond to errors in the estimation~\eqref{eq:finitediff}. Qualitative bounds could be derived by making strong assumptions on $f$, for instance assuming $f(\xx) := f(\xx_0) + \nabla f(\xx_0) (\xx-\xx_0) + \frac{1}{2} (\xx-\xx_0) ^T H (\xx-\xx_0) + O(\norm{(\xx-\xx_0)}^3)$ for every $\xx_0 \in \R^n$, i.e.\ with a Hessian $H$ independed of $\xx_0$. However, this kind of very restricted objective functions are not very interesting in general. We discuss more general functions in Section~\ref{sec:rheconvex} below.

\subsection{Convergence of Random Hessian Estimation}
\label{sec:lev2}
We now derive an exact expression for the expected single step progress of RHE.
\begin{lemma}
\label{lemma:lev1}
Let $B,H \in \Sym_n$ fixed, let $g \colon \Sym_n \to \R$ as in~\eqref{def:g}, let $B_+$ as in~\eqref{eq:leventhalrule} with $\uu \sim \bar{\normal}(\0, I_n)$. Then
\begin{align*}
 \E \left[ g(B_+) \mid B \right] = g(B) - \frac{ 2 g(B) + \Tr[B-H]^2}{n(n+2)}  \leq \left(1 - \frac{2}{n(n+2)} \right) g(B)  \,.
\end{align*}
\end{lemma}
\begin{proof}
By standard calculus and the definition of $g$ we have
\begin{align*}
 g(B_+) = g(B) - (\uu^T(B-H)\uu)^2\,.
\end{align*}
The expectation of the second term on the right hand side was calculated in Lemma~\ref{lemma:factsnormalized}, and the lemma follows from the trivial estimate  $\Tr[B-H]^2 \geq 0$.
\end{proof}

The (uniform) upper bound in Lemma~\ref{lemma:lev1} can be quite far away from the exact value. We estimate with Cauchy-Schwarz $\Tr[B-H]^2 \leq n \norm{B-H}_F^2$ and
\begin{align*}
g(B) - \frac{ 2 g(B) + \Tr[B-H]^2}{n(n+2)}  \geq \left(1 - \frac{2}{n} \right) g(B)  \,.
\end{align*}
Both, the upper and lower bound on the exact factor are tight in general, but they are different by a factor of approximately $n$. Thus one might wonder if the result~\eqref{eq:resultlev} from~\cite{Leventhal:2011} is too conservative in general. But this is not the case, as we answer in the next theorem.

\begin{theorem}
\label{thm:matrix}[Exact RHE]
\index{convergence factor!exact} 
Let $H \in \Sym_n$ fixed, let $\{B_k\}_{k \geq 0}$ a sequence of iterates with $B_{k+1} = B_k + (\uu_k^T (H-B_k)\uu_k) \cdot \uu_k\uu_k^T$ with $\uu_k \sim \bar{\normal}(\0, I_n)$. Denote $X_k := B_k - H$ and let parameters
$\xi_1(k):= \left(\lambda_1^k + \lambda_2^k\right)$, $\xi_2(k):= \left(\lambda_1^k - \lambda_2^k\right)$ with
\begin{align*}
 &\lambda_1 = \frac{2n^2 + 2n - 5 - \omega}{2 n (n+2)} \,,
 &\lambda_2 = \frac{2n^2 + 2n - 5 + \omega}{2 n (n+2)} \,,
\end{align*}
and $\omega = \sqrt{4n^2 + 4n -7}$. Then for $N > 0$
\begin{align*}
 \E\left[\norm{X_N}_F^2\right] &= \xi_1(N) \frac{\norm{X_0}_F^2}{2} - \xi_2(N) \left( \frac{(2n+1) \norm{X_0}_F^2}{2  \omega} - \frac{\Tr[X_0]^2}{\omega} \right) \,, \\
\E\left[\Tr[X_N]^2\right] &= \xi_1(N) \frac{\Tr[X_0]^2}{2} - \xi_2(N) \left( \frac{2 \norm{X_0}_F^2}{\omega} - \frac{(2n+1)\Tr[X_0]^2}{2\omega} \right).
\end{align*}
\end{theorem}
Before going into the proof of this theorem, let us discuss its statement. It is not hard to see, that $\lambda_2  \leq 1- \frac{2}{n(n+2)}$ and $\lambda_1 = 1- \Theta \bigl( \frac{1}{n} \bigr)$. Therefore we can approximate the factors $\xi_1(k) \approx -\xi_2(k) \approx \lambda_2^k$ for $k$ large enough. The upper bound~\eqref{eq:resultlev} from Leventhal and Lewis~\cite{Leventhal:2011} is therefore reached if $\Tr[X_0] = \Tr[B_0 - H] = 0$. However, if $\abs{\Tr[X_0]}$ is large, $\Tr[X_0]^2 = n \norm{X_0}_F^2$, say, then term in the right bracket almost vanishes and $\E[\norm{X_k}_F^2] \approx \frac{1}{2} \lambda_2^k \norm{X_0}_F^2$.
Thus the estimation~\eqref{eq:resultlev} cannot significantly be improved, regardless of $\Tr[X_0]$ we have
$\frac{1}{2} \norm{X_0}_F^2 \bigl(1-\frac{2}{n(n+2)}\bigr)^k \lesssim \E \bigl[ \norm{X_k}_F^2 \bigr] \leq   \norm{X_0}_F^2 \bigl(1-\frac{2}{n(n+2)}\bigr)^k $, where the first inequality holds up to some lower order therms of $n$.

\begin{remark}
\label{rem:markov}
The above theorem derives an exact expression for $\E[ \norm{B_N - H}_F^2$, but no high-probability estimates. With Markov's inequality
one can easily get an upper bound on $\E[ \norm{B_N - H}_F^2$ that holds with high probability. Let $j \leq N$ and $b > 0$ with $(1-\frac{2}{n(n+2)})^{j} = b$. We have $\E[ \norm{B_{N} - H}_F^2 \leq (1-\frac{2}{n(n+2)})^{N-j} \norm{B_{0} - H}_F^2$ with probability at least $1-b$. 
\end{remark}

\begin{proof}[of Thm.~\ref{thm:matrix}]
Let the iteration $k$ be fixed. The exact expression for  the single step progress from Lemma~\ref{lemma:lev1} depends not only on $\norm{X_k}_F^2$, but also on $\Tr[X_k]^2$. Let us also calculate $\E[\Tr[X_{k+1}]^2$. By the definition of the update~\eqref{eq:leventhalrule} we immediately get $\Tr[X_{k+1}] = \Tr[X_k] - \uu_k^T X_k \uu_k$, and therefore
\begin{align*}
 \E \bigl[ \Tr[X_{k+1}]^2 \mid \{X_i\}_{i=0}^k \bigr] &= \Tr[X_k]^2  - \E \left[2 \Tr[X_k] \uu_k^T X_k \uu_k  - (\uu_k^T X_k \uu_k)^2 \mid X_k \right] \\
  &= \left(1- \frac{2n+3}{n (n+2)} \right) \Tr[X_k]^2 + \frac{2}{n(n+2)} \norm{X_k}_F^2 \,,
\end{align*}
with Lemma~\ref{lemma:factsnormalized}. We obtain a linear recurrence for the conditional expectations of $\norm{X_k}_F^2$ and $\Tr[X_k]^2$.
What we now have to do, formally, is to condition on $\{X_{i}\}_{i=0}^{k-1}$ and calculate the expectations again. By the tower property of conditional expectations, $\E[E[\norm{X_{k+1}}_F^2 \mid \{X_{i}\}_{i=0}^{k}]\mid \{X_{i}\}_{i=0}^{k-1}] = \E[\norm{X_{k+1}}_F^2 \mid \{X_{i}\}_{i=0}^{k-1}]$. Repeating this procedure for $\{X_{i}\}_{i=0}^{k-2}$ up to $X_0$, we finally obtain $E[\norm{X_{k+1}}_F^2 \mid X_0] =\E[\norm{X_{k+1}}_F^2]$. We observe that all intermediate expressions only depend linearly on $\norm{X_0}_F^2$ and $\Tr[X_0]^2$, that is we can write $\bigl(\E[\norm{X_k}_F^2, \E[\Tr[X_k]^2] \bigr)^T = C(n)^k \bigl(\norm{X_0}_F^2, \Tr[X_0]^2\bigr)^T$ for a $2 \times 2$ matrix $C(n)$. By linear algebra, we can now decouple the linear recurrence. This is carried out in detail in Lemma~\ref{lemma:2x2} in the appendix.
\end{proof}

%

\subsection{RHE on general strongly convex functions}
\label{sec:rheconvex}
Theorem~\ref{thm:matrix} shows the convergence of RHE to one fixed target matrix $H \in \Pd_n$, where $H=\nabla^2 f(\xx_0)$ is the Hessian of the objective function $f$ at a point $\xx_0 \in \R^n$. For quadratic functions the Hessian $H$ is constant, hence RHE converges to $H$ regardless whether the estimates~\eqref{eq:finitediff} are evaluated at a single point $\xx_0$, or at various different points $\{\xx_k\}_{k \geq 0}$. For an initial estimate $B_0 \in \Pd_n$, at most $O(n^2 \ln \norm{B_0 - H}_F)$ iterations of RHE are necessary to find an approximation $\hat{H} \approx H$, that can be used for the sampling of the search directions in F-RP. Hence the running time of V-RP on quadratic functions is $O(n^2 \ln \norm{B_0 - H}_F + n \ln \frac{1}{\epsilon})$, where $\epsilon > 0$ is the target accuracy. In Section~\ref{sec:imple} we show how this bound can be improved if we only count function evaluations, instead of iterations of RHE or F-RP.

On general strongly convex functions, the Hessian is different at every point in the space. For a fixed point $\xx_0 \in \R^n$, Theorem~\ref{thm:matrix} shows convergence of RHE to $\nabla^2 f(\xx_0)$ if the estimates~\eqref{eq:finitediff} are evaluated at the single point $\xx_0$. However, it might not be useful to compute an approximation of the Hessian at $\xx_0$ if this matrix is not close to the Hessian at the optimum $\xx^* \in \R^n$. Therefore it seems reasonable to interlace the update steps of RHE with the search steps of F-RP, i.e.\ invoke one update step~\eqref{eq:leventhalrule} at each search point $\{\xx_k\}_{k \geq 0}$ that is generated by F-RP. This approach is theoretically justified: As the iterates $\{\xx_k\}_{k \geq 0}$ of F-RP converge (slowly) to $\xx^*$, also the corresponding Hessians $H_k:=\norm{\nabla^2(\xx_k)}$ will converge to $H:= \nabla f(\xx^*)$, and in~\cite[Theorem 2.3]{Leventhal:2011} it is shown, that the Hessian estimates $\{B_k\}_{k \geq 0}$ generated by this interlaced scheme will converge to $H$ as well. However, this theorem does not imply a strong bound on the running time, as their technique only provides a bound on the convergence factor for iterations $k \geq K$, where $K$ is such that $\norm{H_K - H}_F \ll 1$. The following example shows that $K$ can be as large as $O(n \kappaT(H))$. Consider a strongly convex function $f \colon \R^2 \to \R$ with minima $\xx^* = \0$ and Hessians 
\begin{align*}
H(\xx) := \begin{bmatrix} 10^9 + \abs{x_1} & 0\\0 & 1 \end{bmatrix}\,.
\end{align*}
For $\xx \in \R^2$ it is required $\abs{x_1} < 1$ to guarantee $\norm{H(\xx)-H(\xx^*)}_F < 1$. For initial iterate $\xx_0:=(10^{10},0)^T$, F-RP with sampling distribution $\bar{\normal}(\0,I_2)$ needs $O(n \kappaT(H(\xx^*)))$ iterations to find such a close point. On the other hand, $\kappa(H(\xx_0)^{-1} H(\xx^*)) \approx 10$. That is, it would suffice if RHE is only invoked locally at the initial iterate $\xx_0$, because an approximation $\hat{H} \approx H(\xx_0)$ suffices to guarantee fast convergence of F-RP with sampling distribution $\bar{\normal}(\0,\hat{H}^{-1})$. The running time of this approach is only $O(n^2 \ln \norm{B_0 - H(\xx_0)}_F + n \ln \frac{1}{\epsilon})$ instead (see Theorem~\ref{lemma:factor1} below).

We conclude, that the condition $\norm{H_K - H}_F \ll 1$ in~\cite[Theorem 2.3]{Leventhal:2011} is far too strong for what is needed here. The following theorem aims at relaxing this condition, measuring the deviation by $\kappa(H_K^{-1} H)$ instead. That is, we here consider only the situation where the Hessian $\nabla^2 f(\xx_0)$ of the initial iterate $\xx_0$ is already close enough to $\nabla^2 f(\xx^*)$ and RHE is only invoked at $\xx_0$, finding an approximation $\hat{H}$ of $\nabla^2 f(\xx_0)$. We give a bound on the convergence factor of F-RP with sampling distribution $\bar{\normal}(\0,\hat{H}^{-1})$. Using the triangle inequality as in the proof of Theorem 2.3 in~\cite{Leventhal:2011}, it would also be possible to derive an analogous statement for the interlaced V-RP approach.

\begin{theorem}
\label{lemma:factor1}
Let $0 < a \leq b$, $0 \leq c < 1$, $d \geq 1$ and let $B,H,X \in \Pd_n$, with $\norm{Y}_2 \leq b$ and $\norm{Y^{-1}}_2 \leq a^{-1}$ for $Y=\{B,H,X\}$. Here $\norm{Y}_2$ denotes the operator norm induced by the 2-norm. Let $\norm{B-X}_F  \leq a^2b^{-1}c$ and $\kappa(H^{-1} X) \leq d$. Then $\kappa(H^{-1}B) \leq \frac{d+c}{1-c}$.
\end{theorem}

For $f \in C^1_L$ and strongly convex with parameter $M$, we can estimate: $\norm{\nabla^2 f(\xx)}_2 \leq \lambda_{\rm max}(L)$ and $\norm{(\nabla^2 f(\xx))^{-1}}_2 \leq \lambda_{\rm min}^{-1}(M)$, at any $\xx \in \R^n$. Suppose $\xx_0 \in \R^n$ is such that
for $X := \nabla^2 f(\xx_0)$ and $H:=\nabla^2 f(\xx^*)$, $\kappa(H^{-1}B) \leq d$ for some $d \geq 1$ and $B \in \Pd_n$ an initial iterate of RHE. According to Theorem~\ref{lemma:factor1} it takes $O(n^2 (\ln \norm{B-X}_F  + \ln \frac{\lambda_{\rm max}(L)}{\lambda_{\rm min}(M)^2}))$ iterations of (RHE) to find a sufficiently close estimate $B_K$, s.t. $\kappa(B_K^{-1}H) \leq 2d + 1$, say ($c=\frac{1}{2}$).

\begin{proof}[of Theorem~\ref{lemma:factor1}]
We have $\lambda_{\rm max}(X^{-1}H) = \norm{X^{-1} H}_2 \leq \norm{X^{-1}}_2 \norm{H}_2 \leq a^{-1}b$ by submultiplicativity and the assumptions, hence $\lambda_{\rm min}(H^{-1}X) \geq ab^{-1}$. Therefore
\begin{align*}
\lambda_{\rm min}(H^{-1}B) &= \lambda_{\rm min}(H^{-1}X + H^{-1}(B - X)) \\ &\geq \lambda_{\rm min}(H^{-1}X) - \norm{H^{-1}(B - X)}_2 \\ &\geq \frac{a}{b} -  \norm{(B - X)}_F \norm{H^{-1}}_2 \geq \frac{a}{b} - \frac{a^2 c}{b}  \norm{H^{-1}}_2 \,,
\end{align*}
with the assumed upper bound on $\norm{B-X}_F$. As $\norm{H^{-1}}_2 \leq a^{-1}$, we conclude $\lambda_{\rm min}(H^{-1}B) \geq (1-c)ab^{-1} > 0$.  With the analogous argument $\lambda_{\rm max}(H^{-1} B) \leq \lambda_{\rm max}(H^{-1}X) + ab^{-1}c$ and we can estimate the condition number
\begin{align*}
\kappa(H^{-1}B) \leq \frac{\lambda_{\rm max}(H^{-1}X) + \frac{ac}{b}}{\lambda_{\rm min}(H^{-1}X) - \frac{ac}{b}} \leq \frac{d \lambda_{\rm min}(H^{-1}X) + \frac{ac}{b}}{\lambda_{\rm min}(H^{-1}X) - \frac{ac}{b}}\,.
\end{align*}
The fraction $\frac{dx + y}{x-y}$ for $x - y > 0$, $d,y > 0$ is maximized if $x$ is as small as possible. With the lower bound on $\lambda_{\rm min}(H^{-1}X)$ we finally conclude
\begin{align*}
\kappa(H^{-1}B) \leq \frac{ ab^{-1} (d + c) }{ ab^{-1}(1-c)}\,. \tag*{\qed}
\end{align*}
\end{proof}
\subsection{Implementations of RHE}
\label{sec:imple}
Now we proceed to present three implementations of RHE. One difficulty is, that the update~\eqref{eq:leventhalrule} does not guarantee that the matrix $B_+$ is positive definite. An standard result in Wedderburn \cite[pg. 69]{Wedderburn:38} states that for $B\in \Pd_n$, $\uu \in \R^n$ with $\norm{\uu}=1$, the matrix $B + t \uu \uu^T$ is positive definite if $t^{-1} < \uu^T B^{-1} \uu$. Leventhal and Lewis suggest an \emph{ad hoc} projection of $B_+$ onto the cone of $\Pd_n$ matrices. They numerically show that this yields a practicable algorithm \cite{Leventhal:2011}.

\def \sfactor{0.9}
\begin{figure}[htb]
\centering
\null\hfill %
\scalebox{\sfactor}{	
	\begin{minipage}[t]{.51\linewidth}
	      \vspace{0pt}
	      \begin{algorithm}[H]
	     %
	     %
	      \NoCaptionOfAlgo
	      \SetKwInOut{Output}{Output}
	      \Output{Hessian estimate $B_+ \in \Pd_n$}
	      \SetAlgoVlined
	      \DontPrintSemicolon
	      \SetInd{0.3em}{0.5em}
		      \nl $\uu \sim \bar{\normal}(\0, I_n)$\;
		      \nl $\Delta_\uu \leftarrow \frac{f(\xx+\epsilon \uu) - 2f(\xx) + f(\xx-\epsilon \uu)}{\epsilon^2} - \uu^TB\uu$\;
		      \nl \eIf{ $T \leftarrow B + \Delta_\uu \cdot \uu \uu^T \in \Pd_n$}{
		      \nl $B_+ \leftarrow T$\;
		      }{
		      \nl $\vv \leftarrow \mathtt{smallestEVec}(T)$\;
		      \nl $\Delta_\vv \leftarrow \frac{f(\xx+\epsilon \vv) - 2f(\xx) + f(\xx-\epsilon \vv)}{\epsilon^2} -\vv^TT\vv$\;
		      \nl $B_+ \leftarrow \left(B + \Delta_\vv \cdot \vv \vv^T\right) + \Delta_\uu \cdot \uu \uu^T$\;
		      }%
		      \nl \Return{$B_+$}
		      \caption{{$\texttt{updateHessCorr}(f, \xx, B, \epsilon)$}}%
	     \end{algorithm}%
	\end{minipage}  
}%
\hfill\null
\scalebox{\sfactor}{
	\begin{minipage}[t]{.54\linewidth}
	      \vspace{0pt}
	      \begin{algorithm}[H]
	     %
	     %
	      \NoCaptionOfAlgo
	      \SetKwInOut{Requires}{Requires}
	      \SetKwInOut{Output}{Output}
	      \Requires{Persistent storage S of size $O(n^2)$}	      
	      \Output{Hessian estimate $B_+ \in \Pd_n$}
	      \SetAlgoVlined
	      \DontPrintSemicolon
	      \SetInd{0.3em}{0.5em}
	      \SetKwFor{For}{repeat}{}{}
	      \nl $B_+ \leftarrow \text{\texttt{updateHess\{Corr\}}}(f, \xx, B, \epsilon)$\;
	      \nl Add $(\uu, \uu^T \nabla^2 f(\xx) \uu)$, $(\vv, \vv^T \nabla^2 f(\xx)\vv)$ to $S$ \;
	      \nl \If{reuse}{
	      \nl \For{$\text{m times}$}{
	      \nl \ForEach{$(\sve, s) \in S$}{
	      \nl \lIf{$T \leftarrow B_+ + (s-\sve^T B_+ \sve) \cdot \sve\sve^T \in \Pd_n$}{$B_+ \leftarrow T$
	      }
	      }
	      }
	      }
	      \nl \Return{$B_+$}
	      \caption{{$\texttt{updateHessStore}(f, \xx, B, \epsilon, \text{reuse}, m)$}}%
	     \end{algorithm}%
	\end{minipage}
}%
\hfill\null
\caption{Two implementations of RHE~\eqref{eq:leventhalrule}. Left panel: The Hessian estimation $B$ is updated in every step. Positive definiteness is established by an additional projection step. Right panel: the finite difference approximations for $\uu^T \nabla^2 f(\xx) \uu$ are stored. This information can be used for additional update steps that do not require additional function evaluations. The storage $S$ saves the $O(n^2)$ most recently added elements. If the capacity of $S$ is exceeded, the oldest element is deleted (see main text for further information).}
\label{fig:algonew}
\end{figure}

The projection on $\Pd_n$ is only required if the current iterate $B_+$ is needed for the sampling of the search direction, i.e.\ $\uu \sim \bar{\normal}(\0,B_+^{-1})$. The projection step is not necessary if we let the scheme run until it converges to a matrix $\hat{H}$ that is close to the Hessian $H \in \Pd_n$ (this variant is denoted as $\texttt{updateHess}$ in the supporting online material~\cite{Stich:2014online}).

As an alternative to the projection suggested in~\cite{Leventhal:2011}, we would like propose a different one.
In sub-routine $\texttt{updateHessCorr}$ depicted in Figure~\ref{fig:algonew} we ensure that the generated iterates are always positive definite. If $T$ on line~3 is not positive definite (checked by Wedderburn's formula), we apply a second RHE update step in direction $\vv$, where $\vv$ is an eigenvector of $T$ corresponding to the smallest (hence negative) eigenvalue of $B_+$. By standard matrix perturbation theory, as detailed in Lemma~\ref{lemma:pert} below, the twice updated matrix will be positive definite again (as $H$ is). This scheme comes at the expense of two additional function evaluations at $\xx \pm \epsilon \vv$. The updates in line 3 and 7 can directly be implemented using the Sherman–Morrison formula. This version of the VM update has already been successfully used in a recent numerical study~\cite{Stich:2012a}.

\begin{lemma}
\label{lemma:pert}
Let $A \in \Pd_n$, $\xx \in \R^n$ and $\zz_1 \in \R^n$ an eigenvector corresponding to the smallest eigenvalue of $(A - \xx \xx^T)$. Then
\[
 B := A - \xx \xx^T + \abs{\lambda_{\rm min}(A - \xx \xx^T)} \zz_1 \zz_1^T \in {\Pd}_n \,.
\]
\end{lemma}
\begin{proof}
The matrix $(A-\xx \xx^T)$ is symmetric. Let $(A - \xx \xx^T) = \sum_{i=1}^n \lambda_i \zz_i \zz_i^T$ denote its spectral decomposition with $\lambda_1 \leq \lambda_2 \leq \dots \lambda_n$ in increasing order. If $\lambda_1 \geq 0$, then there is nothing to show. Otherwise, we observe that by a variant of Weyl's theorem (cf.~\cite[Theorem~4.3.4]{horn:1985}), $0 \leq \lambda_i (A) \leq \lambda_{i+1}(A - \xx \xx^T) = \lambda_{i+1}$ for $i = 1, \dots, n-1$. Thus at most $\lambda_1$ can be negative. We conclude
\[
 \yy^T B \yy = \yy^T \left( \sum_{i=1}^n \lambda_i \zz_i \zz_i^T + \abs{\lambda_1} \zz_1 \zz_1^T \right) \yy \geq \yy^T \left(\lambda_1 \zz_1 \zz_1^T + \abs{\lambda_1} \zz_1 \zz_1^T \right) \yy \geq 0\,,
\]
for all $\yy \in \R^n$. \qed
\end{proof}

The two implementations discussed so far need in every iteration at least two function evaluations to perform the update. In settings where function evaluations are costly or time consuming one could also store previously computed function values and reuse them for the updates. 
If we assume that either the RHE scheme is invoked locally at one fixed point $\xx_0$ (as discussed in Section~\ref{sec:rheconvex}), or the Hessian $\nabla^2 f(\xx_k)$ is only mildly changing between successive iterates $\xx_k$, then the previously computed values $\uu_{k-t}^T \nabla^2 f(\xx_{k-t}) \uu_{k-t}$ back to some horizon $h$, $t=1,\dots, h$, might still be accurate estimates of the curvature in direction $\uu_{k-t}$ at the current position $\xx_k$. 
Thus, one might apply the update~\eqref{eq:leventhalrule} again for directions $\uu_{k-t}$ using the approximation $\uu_{k-t}^T H(\xx_k) \uu_{k-t} \approx \uu_{k-t}^T H(\xx_{k-t}) \uu_{k-t}$. This requires additional computation time but no additional function evaluations. 
This version of the update is presented in sub-routine $\texttt{updateHessCorr}$ depicted in Figure~\ref{fig:algonew} with with horizon $h=O(n^2)$. 
This variant is motivated by the following observation.

\begin{theorem}
\label{lemma:concentration}
Let $0 < \epsilon < 1$ and constant $C > 0$ large enough, according to~\cite[Theorem 4.2]{Adamczak:2010} (see the proof below). Let $U$ be a set of $h =C n^2$ normalized normal vectors $\{\uu_i\}_{i=1}^h$, $\uu_i \sim \bar{\normal}(\0, I_n)$ for $i=1,\dots, h$ and let $H,B \in \Pd_n$. Then for $\uu$ sampled uniformly at random from the (fixed) set $U$, denoted as $\uu \sim U$, it holds for $B_+ = B + \uu^T (H-B) \uu \cdot \uu \uu^T$:
\begin{align*}
\E_{\uu \sim U} \left[ \norm{B_+ - H}_F^2 \right] \leq \left(1 - \frac{(1-  \epsilon) 2 }{n(n+2)}\right) \norm{B-H}_F^2\,,
\end{align*}
with probability at least $1-e^{-\sqrt{n}}$ over the choice of $U$.
\end{theorem}

\begin{proof}
As in the proof of Lemma~\ref{lemma:lev1} we need to give a lower bound on the expectation $\E_{\uu \sim U}[(\uu^T X \uu)^2]$ for $X = B-H$. Without loss of generality we can assume $\norm{X}_F = 1$. Let $V$ denote an orthogonal matrix such that $VXV^T$ is diagonal, with the vector of eigenvalues $\boldsymbol {\lambda} \in \R^n$ on its diagonal. By considering the set $U' := \{ V\uu \colon \uu \in U\}$ instead of $U$, we can therefore also assume that $X$ is diagonal and write
\begin{align}
 \E_{\uu \sim U'}\left[(\uu^T X \uu)^2\right] =  \E_{\uu \sim U'}\left[ \sum_{i=1}^n \lambda_i^2 u_i^4 + \sum_{i \neq j} \lambda_i \lambda_j u_i^2 u_j^2 \right]\,, \label{eq:est3}
\end{align}
where the subscripts $u_i,\lambda_i$ denote the $i$-th entry of the vectors $\uu$ or $\boldsymbol{\lambda}$. By Lemma~\ref{lemma:factsnormalized} (just set $A=\ee_i \ee_i^T$ and $\xx=\ee_j$, where $\ee_i$ denotes the standard unit vector) we have
\begin{align*}
 \E_{\uu \sim \bar{\normal}(\0,I_n)} [u_i^4] &= \frac{3}{n(n+2)} \,, & \E_{\uu \sim \bar{\normal}(\0,I_n)} [u_i^2 u_j^2] &= \frac{1}{n(n+2)}\,,
\end{align*}
for $i \neq j$. We will now show that the sample approximation, i.e. the expectation over the set $U'$, approximates these values very well with high probability. Theorem 4.2 from~\cite{Adamczak:2010} states that for $\epsilon > 0$, there exist a $C > 0$ (depending only on $\epsilon$), such that with probability at least $1- e^{-\sqrt{n}}$
\begin{align}
 \sup_{\yy \in S^{n-1}} \abs{ \E_{\uu \sim U'} \left[ \lin{\uu,\yy}^4 \right] -  \E_{\uu \sim \bar{\normal}(\0,I_n)} \left[\lin{\uu,\yy}^4\right] } \leq \frac{\epsilon}{n(n+2)} \,, \label{eq:concentrationbd}
\end{align}
if the sample size $h \geq  Cn^2$. Here we have chosen $h$ large enough s.t.~\eqref{eq:concentrationbd} holds with probability at least $1- e^{-\sqrt{n}}$. Hence, the choice $\yy = \ee_i$ gives an estimate of the fourth moment of $u_i$, i.e.
\begin{align}
 \frac{3 - \epsilon}{n(n+2)} \leq  \E_{\uu \sim U'} [u_i^4] \leq\frac{ 3 + \epsilon}{n(n+2)} \,. \label{eq:est1}
\end{align}
The choice $\yy=\frac{1}{\sqrt{2}}(\ee_i \pm \ee_j)$ in~\eqref{eq:concentrationbd} gives us the estimates
\begin{align*}
\frac{12 - 4 \epsilon }{n(n+2)} \leq  \E_{\uu \sim U'} [u_i^4 + u_j^4 + 6 u_i^2 u_j^2 \pm 4 (u_i u_j^3 + u_i^3 u_j) ] \leq \frac{12+4\epsilon}{n(n+2)} \,.
\end{align*}
Adding these two bounds eliminates the last two terms with odd exponents and by subtracting the estimates of $\E_{\uu \sim U'} [u_i^4]$ and $\E_{\uu \sim U'} [u_j^4]$, we get
\begin{align}
\frac{1- \epsilon}{n(n+2)} \leq  \E_{\uu \sim U'} [u_i^2 u_j^2] \leq \frac{1 + \epsilon}{n(n+2)} \,. \label{eq:est2}
\end{align}
Using~\eqref{eq:est1} and~\eqref{eq:est2} in~\eqref{eq:est3}, we get the lower bound
\begin{align*}
 \E_{\uu \sim U'}\left[(\uu^T X \uu)^2\right] \geq (1-\epsilon)  \E_{\uu \sim \bar{\normal}(\0,I_n)}\left[(\uu^T X \uu)^2\right]\,, 
\end{align*}
and the theorem follows from Lemma~\ref{lemma:lev1} which provides a lower bound on $\E_{\uu \sim \bar{\normal}(\0,I_n)}\left[(\uu^T X \uu)^2\right]$. \qed
\end{proof}

\subsection{Computational Experiments}
\label{sec:comp_new}
To complement our theoretical investigation, we conducted a few numerical experiments. We compared the here presented V-RP variants with a number of randomized and deterministic derivative-free algorithms. The set of benchmark functions comprised three quadratic, and one non-convex function.

\subsubsection{Benchmark setting}
The tested variants of V-RP comprise both implementations of RHE that were presented in Figure~\ref{fig:algonew}, in combination with three different implementation of the line search oracle: (i) MATLAB's built-in routine \texttt{fminunc.m}, (ii) an exact line search that needs only two additional function evaluations (three in total on the line) on quadratic functions, by interpolation, and (iii) adaptive step size control from Evolutionary Computation. This last scheme only proves one additional point along the chosen line. The new point is accepted if the function value of the new point is lower than the function value of the previous iterate. In order to ensure that a positive fraction $p$ is accepted on average we use the subroutine \texttt{aSS} detailed in Figure~1 from~\cite{Stich:2012a} with parameters $p=0.27$ and $\sigma = 1$. 
We used the following schemes for our comparison: The Evolution Strategy with Covariance Matrix Adaptation and mirrored sampling and sequential selection (CMA-ES)~\cite{Hansen:2001,Brockhoff:2010}; Implicit Filtering (IMFILL)~\cite{kelly:11}; the classical Down-Hill Simplex algorithm (Nelder-Mead)~\cite{Nelder:1965}, 
 the accelerated version of Nesterov's~\cite{nesterov:randomgradientfree} gradient-free Random Gradient algorithm (Nesterov acc.); Powell's~NEWUOA\cite{powell:06}; and Pattern-Search that is available in MATLAB.
%
%
The full description of the algorithms, as well as the details regarding the parameter selection, we refer to the supporting online material~\cite{Stich:2014online} where the benchmark setting is presented in full detail.

To evaluate the power of adaptation, we tested the algorithms on the following parametric set of functions with increasing curvature. We consider 
\begin{equation*}
f_1(\xx) = \frac{1}{2} \sum_{i=1}^n e^{1 + (i-1) \frac{\log \ell - 1}{n-1}} x_i^2 \,,
\end{equation*}
with parameters $\ell=10^{i}$ for $i=0,\dots, 7$. To test the ``valley-following" abilities of the different algorithms we also include the non-convex Rosenbrock~\cite{Rosenbrock:60} function $f_2$ in the benchmark set:
\begin{equation*}
f_2(\xx) = \sum_{i=1}^{n-1} \left( 100(x_{i+1}-x_{i}^2)^2 +(x_{i}-1)^2 \right)\,.
\end{equation*}
In order to prohibit the tested algorithms from making use of the diagonal structure of the Hessian matrices of $f_1$ we rotate the function domain by generating random rotation matrices $R$ with $RR^T=I_n$ and a shift parameter $\xx_s \sim \normal (0, I_n)$, thus leading to function instances of the form $f(R(\xx-\xx_s))$. We also apply the same transformation and shift to the initial iterate $\xx_0$, which is $\xx_0= \1_n$ for $f_1$ and $\xx_0=\0_n$ for $f_2$.

\subsubsection{Computational results}
We report the average number of function evaluations (FES) needed to reach accuracy $10^{-8}$ on each function (for 31 independent trials). 
A summary of the collected data on $f_1$ for all parameters $\ell$ is presented in Table~\ref{tab:overL}. Table~\ref{tab:newappendix} in the appendix shows more details for $\ell=10^5$.
\begin{table}[tb]
\centering
\scalebox{\sfactor}{
\begin{tabular}{| P{0.6cm} | P{1cm} | P{1cm} | P{1cm} | P{1cm} | P{1cm} | P{1cm} | P{1cm} | P{1cm} |} 
\hline 
$\ell$ & Nesterov Acc. & NEW- UOA & IMFIL & Nelder- Mead & Pattern  & CMA-ES & V-RP (corr) exact & V-RP (store) exact\\ 
\hline  \hline 
0& 3.96  & 0.70  & 4.08 & -& 66.35 & 3.51 & 5.08 & 4.87  \\ \hline 
1& 9.77  & 0.81  & 7.39 & -& 103.61& 3.83 & 6.14 & 5.20  \\ \hline 
2& 37.54 & 2.22  & 8.25 & -& -     & 6.06 & 10.12& 7.70  \\ \hline 
3& 138.25& 6.74  & 10.60& -& -     & 11.20& 17.42& 9.38  \\ \hline 
4& -     & 20.03 & -    & -& -     & 18.56& 27.41& 10.32  \\ \hline 
5& -     & 51.14 & -    & -& -     & 26.45& 38.76& 12.33  \\ \hline 
6& -     & 105.19& -    & -& -     & 34.41& 50.05& 15.52  \\ \hline 
7& -     & -     & -    & -& -     & 42.82& 62.44& 19.44  \\ \hline 
 \end{tabular}

}
\caption{Number of FES/$n^2$ to reach accuracy $10^{-8}$ on $f_1$ with parameter $\ell$, in $n=20$ dimensions (mean over 31 independent runs). A dash `-' indicates that accuracy could not be reached within a budget of 200$n^2$ FES. V-RP equipped with an exact line search oracle.}
\label{tab:overL}
\end{table}
We observe that V-RP-ES (with $\texttt{updateHessStore}$ outperforms all algorithms for $\ell>3$. IMFIL reaches the target accuracy only for $\ell \leq 3$, but in this regime its efficiency is comparable to V-RP. NEWUOA is superior to V-RP for $\ell \leq 3$. For $\ell \geq 4$ NEWUOA needs a rapidly increasing number of FES and can not reach the target accuracy for $\ell=7$. CMA-ES is superior for $\ell \leq 3$ and is outperformed by V-RP for $\ell > 3$. Nelder-Mead fails for all settings to reach the target accuracy (its progress can be observed in Table~\ref{tab:newappendix}. Pattern search is only successful for $\ell \leq 1$ and needs at least a factor of 10 times more FES than all other algorithms. 

The data for Rosenbrock's function is listed in Table~\ref{tab:rosen}. 
\begin{table}[tb]
\centering
\scalebox{\sfactor}{
\begin{tabular}{| p{0.6cm} | P{1cm} | P{1cm} | P{1cm} | P{1cm} | P{1cm} |   P{1cm} | P{1cm} | P{1cm} |} 
\hline 
acc. & F-RP & NEW- UOA & IMFIL & Nelder- Mead & Pattern & CMA-ES & V-RP (corr) ES  & V-RP (store) ES \\ 
\hline  \hline 
$10^{1}$ & 57.21 & 4.73 & -& -& 139.26&  10.36& 24.83& 22.95 \\
$10^{0}$ & 186.64& 10.14& -& -& -     &  26.30& 49.79& 44.52 \\
$10^{-2}$& -     & 12.94& -& -& -     &  31.76& 64.52& 56.18  \\
$10^{-4}$& -     & 14.81& -& -& -     &  33.47& 70.26& 61.02  \\
$10^{-6}$& -     & 16.29& -& -& -     &  34.76& 73.31& 63.23  \\
$10^{-8}$& -     & 17.55& -& -& -     &  35.96& 75.56& 65.16  \\
 \hline \hline 
 sec. & 32.90 & 6.55 & 8588.06 & 14.38 & 1979.45  & 13.96 & 11.62 & 140.89 \\ \hline 
\end{tabular}

}
\caption{Reached accuracy vs. number of FES/$n^2$ on $f_2$ in $n=20$ dimensions (mean over 31 independent runs). A dash `-' indicates that accuracy could not be reached within a budget of 200$n^2$ FES. V-RP equipped with adaptive step size control, see~\cite{Stich:2014online}.}
\label{tab:rosen}
\end{table}
We observe that only V-RP algorithms, CMA-ES, and NEWUOA are able to solve Rosenbrock's function $f_4$ in $n=20$ dimensions. Surprisingly, the NEWUOA outperforms both CMA-ES and all V-RP variants, not only in number of FES but also in computation time. None of the non-adaptive algorithms shows competitive performance.

In summary, our results show that only adaptive schemes like V-RP, CMA-ES and NEWUOA, are competitive algorithms in the presence of ill-conditioning.

\section{Discussion and Conclusion}

\label{sec:discussion}
In this contribution we have analyzed Random Pursuit algorithms that employ (i) a fixed but arbitrary metric (Fixed Metric Random Pursuit) and (ii) a variable metric learning procedure (Variable Metric Random Pursuit). We have detailed convergence proofs and convergence rates for these Random Pursuit algorithms on convex functions. We have used an improved (matrix) quadratic upper bound technique to show expected single-step progress and global convergence of Fixed Metric Random Pursuit on (strictly) convex functions. We have provided exact expressions for the expected progress of the Randomized Hessian estimation scheme (RHE). We have shown that Variable Metric Random Pursuit can achieve almost optimal convergence rate on strongly convex functions that---after a finite learning phase of length at most $O(n^2)$---does not depend on the underlying properties of the unknown Hessian of the function. If the Hessian $H_0$ at the initial search point is close to the Hessian $H$ at the optimum, i.e. $\kappa(H_0^{-1}H) \leq c$ for a constant $c$, it suffices to invoke RHE only once at the beginning.

The numerical experiments show that adaptive schemes are in general (condition number exceeding $10^3$) superior to non-adaptive schemes. For high target accuracy, both V-RP and CMA-ES outperformed the other tested algorithms on the quadratic functions, both in terms of number of FES and time efficiency. NEWUOA shows excellent performance on the non-convex Rosenbrock function. 


A number of theoretical challenges remain. 
For instance, it is still an open question how to analyze Random Pursuit schemes for constrained optimization problems of the form
\begin{align}
 \label{eq:problem_const}
 \min f(x) \quad \text{subject to} \quad x \in \mathcal{C} \, ,
\end{align}
where $\mathcal{C} \subset \R^n$ is a convex set. And it is an open problem to derive convergence guarantees for Random Pursuit schemes on non-convex functions, such as, e.g., on the class of globally convex (or $\delta$-convex) functions \cite{Hu:1989} or on noisy functions with certain bounds on the variance of the noise. Finally, convergence on the important class of non-smooth convex functions is another fundamental challenge for gradient-free optimization that, most likely, needs novel tools and techniques to be developed by the mathematical programming community. 

\subsubsection*{Acknowledgements}
We like to thank the anonymous reviewers whose comments and suggestions very much helped to improve the quality and content of this paper.

\bibliographystyle{spmpsci}      
\bibliography{biblev}   

\appendix
\section{Appendix}

\subsection{Proof of Lemma~\ref{lemma:factsnormalized}}
\begin{proof}
\label{proof:lemma:normalized}
Let $\uu \sim \normal(\0, I_n)$ and let $C \in \Pd_n$ with $C^2 = \Sigma$. The random vector $C\uu$ is $\normal(\0,\Sigma)$ distributed and hence $\ww := C\uu/\norm{C\uu}_{\Sigma^{-1}} = C\uu/\norm{\uu}_2$ has the same distribution as $\vv$ by definition of the normalized distribution. 
Substituting $\vv$ by $\ww$, we obtain expressions that depend only on $\uu$, more precisely, ratios $R_i\bigl( \frac{\uu}{\norm{\uu}_2} \bigr)$ for $i=1,\dotsc,5$ with powers of $\norm{\uu}_2$ in the denominator. For instance, the first and the last one read as:
\begin{align*}
\vv \vv^T = \frac{ C \uu \uu^T C}{\norm{\uu}_2^2} &=: R_1\left( \frac{\uu}{\norm{\uu}_2} \right)\,, &
\norm{\lin{\xx,\vv} \vv}_A^2 = \frac{\norm{\lin{C\xx,\uu} \uu}_{CAC}^2}{\norm{\uu}_2^4} &=: R_5 \left( \frac{\uu}{\norm{\uu}_2} \right)\,.
\end{align*}
Let $S,T  \colon \R^n \to \R$ denote two measurable functions in the random variable $\uu$. We write $S$ and $R$ for short to denote $S(\uu)$ and $T(\uu)$ respectively. Lemma~1 from~\cite{Heijmans:1999} shows that $\E \bigl[\frac{S}{T} \bigr] = \frac{\E[S]}{\E[T]}$ if and only if the covariance $\cov\bigl(\frac{S}{T},T\bigr)= 0$. This follows immediately from $\cov\bigl(\frac{S}{T},T\bigr) = \E[V]- \E \bigl[ \frac{S}{T} \bigr] \E[T]$. We will now apply this result here. The functions $R_i$ for $i=1,\dotsc,5$ do only depend on the direction of the vector $\uu$, but not on its norm. Hence, $R_i$ and $\norm{\uu}_2$ are independent, and especially uncorrelated.

This means that we can calculate the expectations of the numerators and denominators in $R_i$ for $i=1,\dotsc,5$ separately. These values for the numerators follow directly from Lemma~\ref{lemma:factsnormal}, and for the denominators we use $\E\bigl[\norm{\uu}_2^2\bigr] = n$  and $\E\bigl[\norm{\uu}_2^4\bigr] = n(n+2)$, two well-known properties of $\chi^2$-distributed random variables, see e.g.~\cite[Thm.~3.2b.2]{Mathai:1992}. \qed
\end{proof}

The following lemma summarizes some facts about moments of quadratic forms in multivariate normal random variables.

\begin{lemma}
\label{lemma:factsnormal}
Let $\uu \in\normal(\0,\Sigma)$ be drawn from the multivariate normal distribution 
over $\R^n$ with covariance $\Sigma \in \Pd_n$, and let $A \in \Sym_n$ be a symmetric $n \times n$ matrix. Then
\begin{align*}
 \E [ \uu \uu^T ] & = \Sigma\,, &%
 \E [ \uu^T A \uu ] &= \Tr[A \Sigma]\,, &%
 \E [ (\uu^T A \uu)^2 ] = \Tr[A \Sigma]^2 + 2 \Tr[(A\Sigma)^2]\,,
\end{align*}
and for $\xx \in \R^n$,
\begin{align*}
\expect \left[ \lin{\xx,\uu} \uu \right] = \Sigma \xx \, , %
 \quad &\text{and} \quad %
 \expect \left[ \norm{\lin{\xx,\uu} \uu}_A^2 \right] = %
         \Tr[A \Sigma] \norm{\xx}_{\Sigma}^2 %
         + 2 \norm{\xx}_{\Sigma A \Sigma}^2 \, .
\end{align*}
\end{lemma}
The first claim immediately follows from the definition and the second and fourth are consequences of it. This can be seen by applying linearity of expectation to the two identities $\uu^T A \uu = \Tr[\uu\uu^T A]$ and $\lin{\xx,\uu}\uu = \uu\uu^T \xx$. To prove the third and fifth equalities directly, one has again to use linearity of expectation, but also the fourth-moments of normal random variables. We omit the this presentation here, as the claims also follow from~\cite[Thm.~3.2d.3]{Mathai:1992} (with the choice $A_1 = A$ and $A_2 = \xx \xx^T$ for the last claim).

\subsection{Matrix diagonalization}
\begin{lemma}
\label{lemma:2x2}
Let $n \geq 1$ and consider the following $2\times 2$ matrix:
\[
C(n) := \begin{bmatrix}
            1-2\eta & -\eta          \\
            2\eta   & 1- (2n+3)\eta 
          \end{bmatrix} \,,
\]
where $\eta = \frac{1}{n(n+2)}$.
Then
\[
C(n) = 
\begin{bmatrix}
\frac{2n + 1 - \omega}{4 \omega} & \frac{2n + 1 + \omega}{4 \omega} \\
\frac{1}{\omega} & \frac{1}{\omega}
\end{bmatrix}
\begin{bmatrix}
\lambda_1 & 0 \\
0 & \lambda_2 
\end{bmatrix}
\begin{bmatrix}
-2 & \frac{\omega +2n +1}{2}\\
2 & \frac{\omega -2n -1}{2} 
\end{bmatrix}\,,
\]
with $\omega = \sqrt{4n^2 + 4n -7}$,
\begin{align*}
 &\lambda_1 = \frac{2n^2 + 2n - 5 - \omega}{2 n (n+2)} \,,
 &\lambda_2 = \frac{2n^2 + 2n - 5 + \omega}{2 n (n+2)}  \,.
\end{align*}
\end{lemma}
\begin{proof}
The claim can be verified by calculating the product of the three matrices. 
\end{proof}


\subsection{Additional empirical data}
\begin{table}[h!]
\centering
\scalebox{\sfactor}{
\begin{tabular}{| p{0.6cm} | P{1cm} | P{1cm} | P{1cm} | P{1cm} | P{1cm} | P{1cm} | P{1cm} |   P{1cm} | P{1cm} |} 
\hline 
acc. & Nesterov Acc. & NEW- UOA & IMFIL & N-M & Pattern  & CMA-ES & V-RP (corr) exact & V-RP (store) exact \\ 
\hline  \hline 
$10^{4}$ & 2.21 & 0.09 & 1.59 & 0.86 & 2.87 & 0.71 & 0.27 & 0.20 \\
$10^{3}$ & 79.53& 0.33 & 5.95 & 2.83 & 59.55& 2.87 & 1.59 & 1.30 \\
$10^{2}$ & -    & 1.44 & 54.07& 18.63& -    & 7.18 & 6.79 & 5.56  \\
$10^{0}$ & -    & 14.35& -    & -    & -    & 16.35& 22.97& 8.16  \\
$10^{-2}$& -    & 29.21& -    & -    & -    & 21.87& 30.26& 9.20  \\
$10^{-4}$& -    & 38.96& -    & -    & -    & 24.50& 34.22& 10.25 \\
$10^{-6}$& -    & 45.76& -    & -    & -    & 25.52& 36.76& 11.26 \\
$10^{-8}$& -    & 51.14& -    & -    & -    & 26.45& 38.76& 12.33 \\
 \hline \hline 
 sec. & 21.96 & 50.67 & 8273.24 & 16.28 & 1863.09 & 10.08 & 7.15 & 26.35 \\ \hline 
\end{tabular}

}
\caption{Reached accuracy vs. number of FES/$n^2$ on $f_1$ with parameter $\ell = 10^5$ in $n=20$ dimensions (mean over 31 independent runs). A dash `-' indicates that accuracy could not be reached within a budget of 200$n^2$ FES. Average computation time of a single run on a single core CPU: either the time until the budget of 200$n^2$ FES is exceed or the time needed to reach accuracy $10^{-8}$.}
\label{tab:newappendix}
\end{table}

\cleardoublepage
\thispagestyle{plain}

\begin{center}
\textsc{Supporting Online Material for}\\[2em]
{ {\Large {{Variable Metric Random Pursuit}$^*$}}}
\end{center}
~\\
\begin{center}
\textsc{S.~U. Stich$^\dagger$, C.~L. M\"{u}ller$^\ddagger$, and B. G\"{a}rtner$^\S$}
\end{center}

\newpage
\section{Computational Experiments}
\label{sec:numerical}
In \cite{Stich:2012a} we have already presented extensive numerical results of V-RP in comparison with other randomized variable metric schemes. There, we analyzed the influence of the Hessian eigenvalue )]spectrum on the convergence of these schemes. The main result from these tests were that among a parametrized set of Hessian matrices with equal trace and condition number, the matrices with a sigmoidal spectrum are the most difficult to learn for the RHE update scheme and matrices with an inverse sigmoidal (almost flat) distribution of eigenvalues are easier to learn.

We here compare the performance of V-RP with a number of randomized and \emph{deterministic} derivative-free algorithms. The set of test functions comprises three quadratic functions (including $f_1$) with different spectra and one non-convex function ($f_2$). We first present the definition of the test functions and describe the numerical performance evaluation protocol. We then detail the algorithms and their parametrization.

\subsection{Benchmark Functions}
\label{sec:benchfun}
The first two functions are $f_1$ and $f_2$ (see definition in Section~\ref{sec:comp_new}). 
We consider two additional quadratic functions with parameter $\ell \geq 1$.
\begin{align*}
f_3(\xx) &= \frac{1}{2} \left( \sum_{i=1}^{\lceil\frac{n}{2}\rceil} x_i^2 + \ell \sum_{i= \lfloor\frac{n}{2} \rfloor}^{n} x_i^2 \right)\,, & 
f_4(\xx) &= \frac{1}{2} \left( x_1^2 + \frac{\ell}{2} \sum_{i=2}^{n-1}  x_i^2 +  \ell x_n^2 \right)\,,
\end{align*}
The Hessian matrices in both functions have the same maximal ($\ell$) and minimal ($1$) eigenvalues. The function $f_3$ has two different scales that are distributed evenly among the dimensions. The second function $f_4$ has -- for large dimension -- one global scale with one small and one large eigenvalue. A previous numerical study \cite{Stich:2012a} suggests that function $f_3$ is challenging for RP algorithms and $f_4$ is easy among all convex quadratic functions with the same condition number and trace.

 The functions $f_3$ and $f_4$ are limit cases of these function classes and can be considered as worst and best cases. 

The quadratic functions attain their minimum function value at $\xx^*=\0_n$, the all zero vector ($f_1(\0_n)=f_3(\0_n)=f_4(\0_n)=0$). The Rosenbrock function is minimized at $\xx^* = \1_n$, with $f_2(\1_n)=0$.

In order to prohibit the tested algorithms from making use of the diagonal structure of the Hessian matrices of $f_1$, $f_3$ and $f_4$, we rotate the function domain by generating random rotation matrices $R$ with $RR^T=I_n$ and a shift parameter $\xx_s \sim \normal (0, I_n)$, thus leading to function instances of the form
\[
 f(R(\xx-\xx_s)) \, .
\]
We also apply the same transformation and shift to the initial iterate $\xx_0$. This procedure  and/or the special structure of $\xx^*$. We use as initial iterate $\xx_0= \1_n$ for the quadratic functions and $\xx_0=\0_n$ for $f_2$.

\subsection{Algorithms}
\subsubsection{V-RP schemes}
\label{sec:lsimplement}
We implemented F-RP with fixed covariance $\Sigma=I_n$. This scheme is simply referred to as RP. We also implemented two RHE schemes presented in this paper, \texttt{updateHessCorr} and \texttt{updateHessStore}. The parameter setting for the latter one can be found in Fig.~\ref{fig:algonew} with $m=10$. The choice of the parameter $\epsilon$ does not influence the performance of the schemes on the quadratic functions. We thus set it to $\epsilon = 1$. For $f_4$ we used $\epsilon = 10^{-6}$. For \texttt{updateHessStore} we apply the updates from the storage every $n$-th iteration in random order, starting after the $n^2$-th iteration (as soon as enough data is collected).
All RP schemes have to be combined with an implementation of the line search oracle. We tested three different implementations and present them here in increasing order of function evaluations they consume.

\paragraph{ES:} This scheme is also know as (1+1)-Evolution Strategy (ES). To perform the line search, only one additional point along the chosen line is probed. The scheme accepts the new point if the function value of the new point is lower than the function value of the previous iterate. In order to ensure that a positive fraction $p$ is accepted on average we use an adaptive step size scheme (aSS) as detailed in the procedure \texttt{aSS} in Figure~1 from~\cite{Stich:2012a} with parameters $p=0.27$ and $\sigma = 1$.

\paragraph{Exact:} By probing two additional points on the given line (three in total), the exact minimizer can be computed if the function is quadratic ($f_1$,$f_3$,$f_4$). For $f_2$ this scheme may fail to report a better value but we observed in our experiments that the quality of the guessed minimizer is sufficient.

\paragraph{Matlab:} We use the built-in MATLAB routine \texttt{fminunc.m} from the optimization toolbox~\cite{MATLAB:2012} with \texttt{optimset('TolX'=$0.01$)} as approximate line search. In the present gradient-free setting \texttt{fminunc.m} uses a mixed cubic/quadratic polynomial line search where the first three points bracketing the minimum are found by bisection~\cite{MATLAB:2012}.

\subsubsection{CMA-ES}
The Evolution Strategy with Covariance Matrix Adaptation~\cite{Hansen:2001} (CMA-ES) is one of the most popular and efficient schemes for derivative free optimization on non-convex and noisy problems. New search points are sampled from a multivariate normal distribution whose parameter are updated in each iteration. The fundamental design principle used here is slightly different than for the V-RP schemes.  Instead of performing the updates on an estimation of the Hessian (and then computing its inverse), the updates are performed directly on the inverse directly.
The CMA-ES scheme is augmented by an auxiliary variable called evolution path that takes into account the correlation of successive means taken over a finite horizon. This is similar in spirit to Rao-Blackwellization techniques in Marko Chain Monte Carlo methods \cite{Andrieu:2008} and Polyak's heavy ball method in first-order optimization \cite{Polyak:1987}.

Among the many different instances of CMA-ES, we consider here the one that is the fastest scheme for quadratic functions known today. This scheme is called the (1,4)-CMA-ES with mirrored sampling and sequential selection. We also refer to~\cite{Brockhoff:2010} for a full description of this scheme and all parameter settings used. The scale parameter is set to $\sigma=1$ for our experiments. The code for the (1,4)-CMA-ES scheme has been retrieved from \texttt{http://coco.gforge.inria.fr/doku.php?id=bbob-2010-results}.

\subsubsection{Nesterov's Random Gradient schemes}
Nesterov~\cite{nesterov:randomgradientfree} introduced a derivative-free optimization scheme that is very similar to RP. Optimization is performed iteratively among randomly chosen lines. The optimal step size is estimated by finite-difference estimation. This scheme is called Random Gradient (RG) method. Its advantage over RP is that the finite-difference calculation needs only one additional function evaluation, and it is guaranteed to make progress in every iteration (opposed to the ES line search). One disadvantage is that the RG method needs an estimate of the curvature of the function which is not available in practice. For test purposes, we always use the correct curvature of the objective function (parameter $\ell$) as input to the RG scheme.

Similar to the accelerated gradient methods for convex optimization, an accelerated version of the RG scheme is available~\cite{nesterov:randomgradientfree}. This scheme also needs only two function evaluation per iteration and shows superior theoretical convergence properties \cite{nesterov:randomgradientfree}.

\subsubsection{Pattern Search}
Pattern Search is a deterministic scheme that evaluates the objective function in every iteration on $2n$ predefined points on a stencil. We use the built-in MATLAB routine \texttt{patternsearch} from the Global Optimization Toolbox~\cite{MATLAB:2012} with parameters \texttt{Cache}$=$\texttt{on}, \texttt{InitialMeshSize}=$1$, \texttt{TolMesh}=$1e$-$20$, \texttt{TolX}=$1e$-$20$, \texttt{TolFun}=$1e$-$20$.

\subsubsection{Nelder-Mead}
We use the built-in MATLAB routine \texttt{fminsearch} from the Optimization Toolbox~\cite{MATLAB:2012} which implements the classical Nelder-Mead (N-M) Down-Hill Simplex algorithm~\cite{Nelder:1965}. We use the algorithm with parameters \texttt{TolX}=$1e$-$20$, \texttt{TolFun}=$1e$-$20$.

\subsubsection{NEWUOA}
NEWUOA~\cite{powell:06} is an iterative algorithm that builds a quadratic model of the objective function. Steps are proposed by minimizing this model within a trust region. When the quadratic model is updated, the new model interpolates the objective function in \texttt{npt} points, typically \texttt{npt}=$2n+1$. We use the C implementation made available by M. Guilbert on \texttt{http://www.inrialpes.fr/bipop/ people/guilbert/newuoa/newuoa.html}. This code is based on the original FORTRAN implementation of NEWUOA by Powell. We use the standard setting \texttt{npt}=$2n+1$ and $\rho_{\rm beg} = 1 $, $\rho_{\rm end} = 10^{-14}/\ell$.

\subsubsection{Implicit Filtering}
Implicit filtering (IMFIL) is a hybrid of a Quasi-Newton and a VM scheme. The gradients and Hessians are approximated by finite differences.
We use the MATLAB code by Kelly~\cite{kelly:11}, available on \texttt{http://www4.ncsu.edu/\textasciitilde ctk/imfil.html} with the setting \texttt{smooth\_problem}=$1$ and \texttt{bscales}$=(1, 2^{-1}, \dots, 2^{-100})$ to avoid premature convergence.

\subsection{Convergence on the function pair $f_3/f_4$}
We test the convergence of all algorithms on $f_3$ and $f_4$ for dimension $n=20$. We performed 31 independent trials of the same experiment. We let the algorithms run until either the accuracy $10^{-8}$ was successfully reached, or a budget of total 200$n^2$ function evaluations (FES) was consumed. In addition to the number of FES we also recorded the run time (in seconds) needed to perform a single trial. All algorithms where executed on a single core. We report the number of function evaluations performed by the algorithm to reduce the function value by one order of magnitude. 

We first focus on the results for function $f_1$ as it presents a kind of worst-case scenario for adaptive schemes. The data for the benchmark set is listed in Table~\ref{tab:appendix_twosphere}. Table~\ref{tab:sec6_1} shows a subset of these data, neglecting non-adaptive schemes as well as some combinations of V-RP and line search implementation. The data in Table~\ref{tab:sec6_1} are graphically depicted in Fig.~\ref{fig:sec6_1}. We observe that among the successful algorithms (CMA-ES, V-RP ES, and V-RP Exact) CMA-ES needed about a factor of 3.4 more FES to reach accuracy $10^{-9}$ than both V-RP schemes but only needs half the run time. The other four algorithms (NEWUOA, IMFIL, Nelder-Mead, Pattern search) only managed to reach accuracy $10^{0}$ - $10^{2}$ with the same budget of FES. With the exception of Nelder-Mead their execution time is exceeding the time of CMA-ES by a factor of 14--190. 

We also observe that all seven tested algorithms make rapid progress at the beginning (up to accuracy roughly $10^{2}$). They then get either stuck or---after a learning phase---resume fast convergence toward higher accuracy levels (roughly $10^{-2}$-$10^{-8}$). These phases are typical for these kind of algorithms (see the  discussion in Section~\ref{sec:numericssec5} below).

\def \sfactor{0.89}
\begin{table}
\scalebox{\sfactor}{
\begin{tabular}{| p{0.6cm} | P{1.3cm} | P{1.3cm} | P{1.3cm} | P{1.3cm} | P{1.3cm} | P{1.3cm} | P{1.3cm} |} 
\hline 
acc. & NEWUOA & IMFIL & N-M & Pattern & CMA-ES & V-RP \newline ES & V-RP Exact \\ 
\hline  \hline 
$10^{7}$ & 0.06  & 1.58  & 0.74  & 2.18  & 0.23  & 0.68  & 0.11  \\
 $10^{6}$ & 0.13  & 2.32  & 1.72  & 6.52  & 0.57  & 1.21  & 0.30  \\
 $10^{5}$ & 0.17  & 3.07  & 2.52  & 10.63  & 0.93  & 1.77  & 0.53  \\
 $10^{4}$ & 0.23  & 3.91  & 3.19  & 15.32  & 1.28  & 2.29  & 0.80  \\
 $10^{3}$ & 0.28  & 4.68  & 3.71  & 19.79  & 1.62  & 2.83  & 1.09  \\
 $10^{2}$ & 0.34  & 5.59  & 4.11  & 25.18  & 2.27  & 3.37  & 1.40  \\
 $10^{1}$ & 0.40  & - & 6.58  & - & 15.29  & 6.15  & 2.97  \\
 $10^{0}$ & 175.60  & - & - & - & 30.44  & 11.45  & 13.89  \\
 $10^{-1}$ & - & - & - & - & 45.58  & 13.75  & 17.52  \\
 $10^{-2}$ & - & - & - & - & 55.81  & 14.74  & 19.25  \\
 $10^{-3}$ & - & - & - & - & 62.74  & 15.68  & 20.07  \\
 $10^{-4}$ & - & - & - & - & 66.43  & 16.61  & 20.61  \\
 $10^{-5}$ & - & - & - & - & 69.65  & 17.49  & 21.11  \\
 $10^{-6}$ & - & - & - & - & 70.65  & 18.49  & 21.65  \\
 $10^{-7}$ & - & - & - & - & 71.81  & 19.42  & 22.17  \\
 $10^{-8}$ & - & - & - & - & 72.23  & 20.37  & 22.75  \\
 \hline \hline 
 sec. & 438.34 & 5711.35 & 10.70 & 1904.77 & 29.63 & 45.37 & 35.92 \\ 
 \hline 
\end{tabular}

}
\caption{Accuracy vs. number of FES/$n^2$ on $f_3$ with parameter $\ell=10^{7}$ in $n=20$ dimensions (mean over 31 independent runs). A dash `-' indicates that accuracy could not be reached within a budget of 200$n^2$ FES. V-RP is implemented with update scheme \texttt{updateHessStore} for two different implementations of the line search (ES and Exact). Average computation time of a single run on a single core CPU; the time until the budget of 200$n^2$ FES is exceeded or the time needed to reach accuracy $10^{-8}$.
}
\label{tab:sec6_1}
\end{table}

\begin{figure}[ht]
\centering
\includegraphics[width=.95\textwidth]{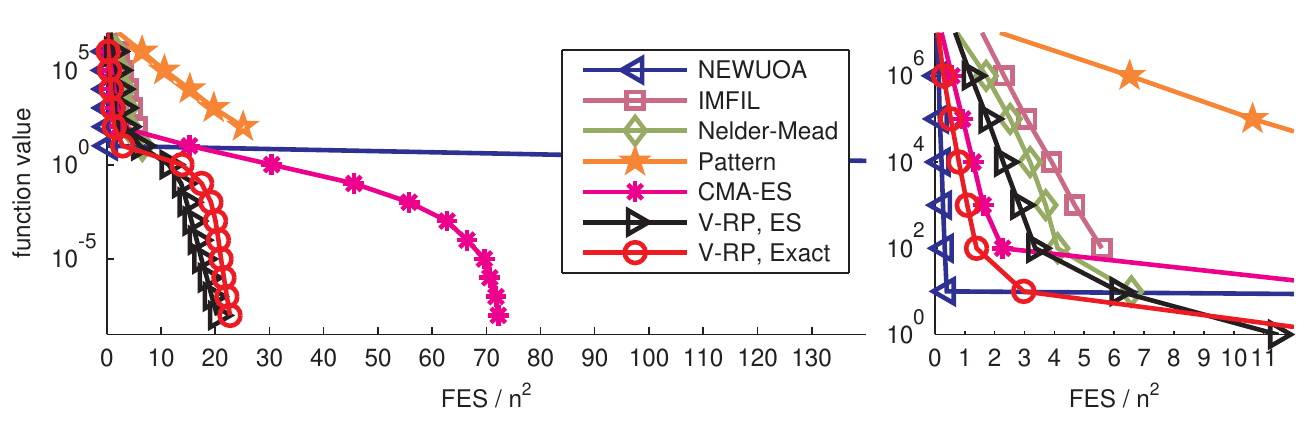}
\caption{Reached accuracy vs. number of FES/$n^2$ on $f_3$ with parameter $\ell=10^{7}$ in $n=20$ dimensions (mean over 31 independent runs). See Tables~\ref{tab:sec6_1} and~\ref{tab:appendix_twosphere}. V-RP is implemented with update scheme \texttt{updateHessStore} for two different implementations of the line search (ES and Exact).}
\label{fig:sec6_1}
\end{figure}

The empirical results on function $f_4$ reveal several interesting features. NEWUOA, CMA-ES, and V-RP all show faster convergence on this function compared to $f_3$ (in accordance with previous experiments \cite{Stich:2012a}). These schemes also solve the problem to high accuracy. All other algorithms show, however, reduced performance on $f_4$ when compared to $f_3$ (see data in Tables~\ref{tab:appendix_twosphere} and ~\ref{tab:appendix_flatsphere}). Both Pattern Search and Nesterov's schemes do not reach an accuracy below $10^5$. RP ES, IMFIL, and N-M need a considerably higher number of FES to reach an accuracy of $10^3$. 

In summary, these results show that only adaptive schemes and, to some extent, NEWUOA, are competitive algorithms in the presence of ill-conditioning. The results also suggest that the performance of popular methods such as Implicit Filtering and Nelder-Mead (the standard derivative-free method in MATLAB) are not suitable \emph{even for problems where only a few dimensions are not on the same scale} (such as $f_4$).  

\subsection{Evaluating the power of adaptation}
Given practical limitations on the available budget of function evaluations it is natural to ask whether function evaluations should be rather spent on estimating the Hessian or for direct function optimization. In order to evaluate the power of adaptation we test the described algorithms on Rosenbrock's function and the following parametric set of functions with increasing curvature. We consider $f_1$ with parameters $\ell=10^{i}$ for $i=0,\dots, 7$. For $i=0$ this function equals the so-called sphere function $f(\xx)=\frac{1}{2} \xx^T \xx$. We report the average number of FES needed to reach accuracy $10^{-8}$ on each function (for 31 independent trials). The data of the full benchmark set are listed in Tables~\ref{tab:appendix_linexp0}-\ref{tab:appendix_linexp7}. Table~\ref{tab:overL} shows a subset of algorithms including Nesterov's accelerated RG scheme.

We observe that V-RP-ES (with $\texttt{updateHessStore}$ and line search ES) outperforms all algorithms for $\ell>3$. Nesterov's non-adaptive accelerated RG scheme is superior to V-RP only for the isotropic case $\ell=0$. IMFIL reaches the target accuracy only for $\ell \leq 3$, but in this regime it is more efficient than V-RP. NEWUOA is superior to V-RP for $\ell \leq 3$. For $\ell \geq 4$ NEWUOA needs a rapidly increasing number of FES and can not reach the target accuracy for $\ell=7$. CMA-ES is superior for $\ell \leq 3$ and is outperformed by V-RP for $\ell > 3$. Nelder-Mead fails for all settings to reach the target accuracy. 
Pattern search is only successful for $\ell \leq 1$ and needs at least a factor of 10 times more FES than all other algorithms. 

Finally, only the V-RP algorithms, CMA-ES, and NEWUOA are able to solve Rosenbrock's function $f_2$ in $n=20$ dimensions (see Table~\ref{tab:rosen} for all data). 
Surprisingly, the NEWUOA outperforms both CMA-ES and all V-RP variants. None of the non-adaptive algorithms shows competitive performance. Pattern search and standard RP with ES line search reach an accuracy of $10^1$ and $10^0$, respectively. Implicit filtering, Nelder-Mead, and Nesterov's schemes even fail to get an accuracy of $10^1$.

In summary, the empirical results from the presented benchmark clearly show the superiority of adaptive schemes such as V-RP and CMA-ES.

\def \sfactor{0.65}
\begin{table}
\hspace{-1.8cm}
\scalebox{\sfactor}{
\begin{tabular}{| p{0.6cm} | P{1.3cm} | P{1.3cm} | P{1.3cm} | P{1.3cm} | P{1.3cm} | P{1.3cm} | P{1.3cm} |   P{1.3cm} | P{1.3cm} | P{1.3cm} | P{1.3cm} | P{1.3cm} | P{1.3cm} | P{1.3cm} |} 
\hline 
acc. & RP \newline ES & NEWUOA & IMFIL & N-M & Pattern & Nesterov RG & Nesterov Acc. & CMA-ES & V-RP (corr) ES & V-RP (corr) Exact & V-RP (corr) matlab & V-RP (store) ES & V-RP (store) Exact & V-RP (store) matlab \\ 
\hline  \hline 
$10^{7}$ & 0.10  & 0.06  & 1.58  & 0.74  & 2.18  & 0.36  & 0.27  & 0.23  & 0.72  & 0.19  & 0.35  & 0.68  & 0.11  & 0.22  \\
 $10^{6}$ & 0.25  & 0.13  & 2.32  & 1.72  & 6.52  & 0.90  & 0.92  & 0.57  & 1.27  & 0.61  & 1.28  & 1.21  & 0.30  & 0.74  \\
 $10^{5}$ & 0.40  & 0.17  & 3.07  & 2.52  & 10.63  & 1.47  & 2.54  & 0.93  & 1.80  & 0.93  & 2.08  & 1.77  & 0.53  & 1.37  \\
 $10^{4}$ & 0.55  & 0.23  & 3.91  & 3.19  & 15.32  & 2.04  & 3.91  & 1.28  & 2.30  & 1.24  & 2.88  & 2.29  & 0.80  & 1.97  \\
 $10^{3}$ & 0.71  & 0.28  & 4.68  & 3.71  & 19.79  & 2.60  & 5.05  & 1.62  & 2.80  & 1.55  & 3.59  & 2.83  & 1.09  & 2.64  \\
 $10^{2}$ & 0.88  & 0.34  & 5.59  & 4.11  & 25.18  & 3.18  & - & 2.27  & 3.37  & 1.84  & 4.51  & 3.37  & 1.40  & 3.49  \\
 $10^{1}$ & - & 0.40  & - & 6.58  & - & 3.92  & - & 15.29  & 17.93  & 11.47  & 32.47  & 6.15  & 2.97  & 12.42  \\
 $10^{0}$ & - & 175.60  & - & - & - & - & - & 30.44  & 34.11  & 44.37  & 86.42  & 11.45  & 13.89  & 30.24  \\
 $10^{-1}$ & - & - & - & - & - & - & - & 45.58  & 40.75  & 52.29  & 101.82  & 13.75  & 17.52  & 36.26  \\
 $10^{-2}$ & - & - & - & - & - & - & - & 55.81  & 45.26  & 59.92  & 113.60  & 14.74  & 19.25  & 39.83  \\
 $10^{-3}$ & - & - & - & - & - & - & - & 62.74  & 48.28  & 65.66  & 121.34  & 15.68  & 20.07  & 41.49  \\
 $10^{-4}$ & - & - & - & - & - & - & - & 66.43  & 50.45  & 69.75  & 127.00  & 16.61  & 20.61  & 42.54  \\
 $10^{-5}$ & - & - & - & - & - & - & - & 69.65  & 52.16  & 72.68  & 131.63  & 17.49  & 21.11  & 43.62  \\
 $10^{-6}$ & - & - & - & - & - & - & - & 70.65  & 53.72  & 74.87  & 135.02  & 18.49  & 21.65  & 44.86  \\
 $10^{-7}$ & - & - & - & - & - & - & - & 71.81  & 54.99  & 76.78  & 137.87  & 19.42  & 22.17  & 46.00  \\
 $10^{-8}$ & - & - & - & - & - & - & - & 72.23  & 56.15  & 78.39  & 140.34  & 20.37  & 22.75  & 47.10  \\
 \hline \hline 
 sec. & 30.77 & 438.34 & 5711.35 & 10.70 & 1904.77 & 18.26 & 22.68 & 29.63 & 16.83 & 18.65 & 85.60 & 45.37 & 35.92 & 521.82 \\ \hline 
\end{tabular}

}
\caption{Reached accuracy vs. number of FES/$n^2$ on $f_3$ with parameter $\ell=10^{7}$ in $n=20$ dimensions (mean over 31 independent runs). A dash `-' indicates that accuracy could not be reached withing a budget of 200$n^2$ FES. See main text for further information.}
\label{tab:appendix_twosphere}
\end{table}
\begin{table}
\hspace{-1.8cm}
\scalebox{\sfactor}{
\begin{tabular}{| p{0.6cm} | P{1.3cm} | P{1.3cm} | P{1.3cm} | P{1.3cm} | P{1.3cm} | P{1.3cm} | P{1.3cm} |   P{1.3cm} | P{1.3cm} | P{1.3cm} | P{1.3cm} | P{1.3cm} | P{1.3cm} | P{1.3cm} |} 
\hline 
acc. & RP \newline ES & NEWUOA & IMFIL & N-M & Pattern & Nesterov RG & Nesterov Acc. & CMA-ES & V-RP (corr) ES & V-RP (corr) Exact & V-RP (corr) matlab & V-RP (store) ES & V-RP (store) Exact & V-RP (store) matlab \\ 
\hline  \hline 
 $10^{6}$ & 0.01  & 0.05  & 1.26  & 0.29  & 0.01  & 0.35  & 0.25  & 0.01  & 0.02  & 0.01  & 0.02  & 0.02  & 0.01  & 0.02  \\
 $10^{5}$ & 0.05  & 0.08  & 4.45  & 0.35  & 0.71  & 0.94  & 0.42  & 1.06  & 0.10  & 0.10  & 0.08  & 0.11  & 0.03  & 0.20  \\
 $10^{4}$ & 8.54  & 1.69  & 37.59  & 4.33  & - & - & - & 4.97  & 3.79  & 2.48  & 5.08  & 3.99  & 2.41  & 4.88  \\
 $10^{3}$ & 19.20  & 4.03  & 144.48  & 72.65  & - & - & - & 6.92  & 7.40  & 5.82  & 11.90  & 7.49  & 5.65  & 12.21  \\
 $10^{2}$ & 29.92  & 6.63  & - & - & - & - & - & 8.50  & 10.20  & 9.08  & 18.22  & 10.59  & 9.16  & 20.04  \\
 $10^{1}$ & 41.29  & 8.95  & - & - & - & - & - & 10.11  & 12.93  & 12.54  & 24.20  & 12.52  & 12.43  & 26.54  \\
 $10^{0}$ & - & 11.78  & - & - & - & - & - & 14.99  & 21.38  & 31.35  & 45.85  & 13.64  & 16.29  & 33.48  \\
 $10^{-1}$ & - & 51.82  & - & - & - & - & - & 18.59  & 32.83  & 43.32  & 72.59  & 14.59  & 18.43  & 37.34  \\
 $10^{-2}$ & - & 82.25  & - & - & - & - & - & 19.48  & 36.66  & 51.46  & 89.86  & 15.52  & 19.19  & 38.63  \\
 $10^{-3}$ & - & 104.93  & - & - & - & - & - & 20.03  & 38.28  & 54.49  & 97.54  & 16.53  & 19.72  & 39.72  \\
 $10^{-4}$ & - & 113.03  & - & - & - & - & - & 20.56  & 39.45  & 55.86  & 100.07  & 17.49  & 20.25  & 40.86  \\
 $10^{-5}$ & - & 117.40  & - & - & - & - & - & 21.10  & 40.56  & 56.95  & 102.38  & 18.43  & 20.80  & 41.98  \\
 $10^{-6}$ & - & 120.75  & - & - & - & - & - & 21.61  & 41.53  & 57.76  & 104.36  & 19.35  & 21.30  & 43.15  \\
 $10^{-7}$ & - & 123.84  & - & - & - & - & - & 22.13  & 42.52  & 58.49  & 105.94  & 20.34  & 21.82  & 44.24  \\
 $10^{-8}$ & - & 126.26  & - & - & - & - & - & 22.60  & 43.55  & 59.24  & 107.37  & 21.25  & 22.36  & 45.36  \\
 \hline \hline 
 sec. & 30.76 & 274.79 & 7976.66 & 17.38 & 1926.26 & 17.36 & 22.94 & 8.90 & 13.13 & 13.86 & 85.21 & 46.66 & 37.94 & 528.10 \\ \hline 
\end{tabular}

}
\caption{Reached accuracy vs. number of FES/$n^2$ on $f_4$ with parameter $\ell=10^{7}$ in $n=20$ dimensions (mean over 31 independent runs). A dash `-' indicates that accuracy could not be reached within a budget of 200$n^2$ FES. See main text for further information.}
\label{tab:appendix_flatsphere}
\end{table}
\begin{table}
\hspace{-1.8cm}
\scalebox{\sfactor}{
\begin{tabular}{| p{0.6cm} | P{1.3cm} | P{1.3cm} | P{1.3cm} | P{1.3cm} | P{1.3cm} | P{1.3cm} | P{1.3cm} |   P{1.3cm} | P{1.3cm} | P{1.3cm} | P{1.3cm} | P{1.3cm} | P{1.3cm} | P{1.3cm} |} 
\hline 
acc. & RP \newline ES & NEWUOA & IMFIL & N-M & Pattern & Nesterov RG & Nesterov Acc. & CMA-ES & V-RP (corr) ES & V-RP (corr) Exact & V-RP (corr) matlab & V-RP (store) ES & V-RP (store) Exact & V-RP (store) matlab \\ 
\hline  \hline 
 $10^{1}$ & 0.07  & 0.06  & 0.23  & 0.42  & 3.16  & 0.08  & 0.07  & 0.24  & 0.19  & 0.11  & 0.16  & 0.21  & 0.10  & 0.18  \\
 $10^{0}$ & 0.35  & 0.15  & 0.66  & 5.10  & 10.06  & 0.45  & 0.44  & 0.59  & 1.13  & 0.64  & 1.14  & 1.14  & 0.57  & 1.02  \\
 $10^{-1}$ & 0.67  & 0.22  & 1.09  & 46.34  & 16.86  & 0.87  & 0.86  & 0.95  & 2.04  & 1.19  & 2.28  & 2.06  & 1.10  & 2.09  \\
 $10^{-2}$ & 0.99  & 0.30  & 1.51  & - & 23.58  & 1.35  & 1.29  & 1.30  & 2.98  & 1.73  & 3.38  & 3.02  & 1.64  & 3.14  \\
 $10^{-3}$ & 1.32  & 0.37  & 1.96  & - & 30.87  & 1.85  & 1.72  & 1.66  & 3.90  & 2.31  & 4.44  & 3.93  & 2.19  & 4.16  \\
 $10^{-4}$ & 1.64  & 0.44  & 2.35  & - & 37.27  & 2.35  & 2.17  & 2.03  & 4.88  & 2.90  & 5.61  & 4.81  & 2.71  & 5.26  \\
 $10^{-5}$ & 1.97  & 0.51  & 2.80  & - & 44.71  & 2.86  & 2.62  & 2.39  & 5.80  & 3.46  & 6.80  & 5.78  & 3.24  & 6.42  \\
 $10^{-6}$ & 2.32  & 0.58  & 3.23  & - & 52.74  & 3.37  & 3.05  & 2.76  & 6.75  & 3.99  & 7.94  & 6.69  & 3.79  & 7.64  \\
 $10^{-7}$ & 2.66  & 0.64  & 3.66  & - & 59.56  & 3.88  & 3.53  & 3.13  & 7.67  & 4.54  & 9.08  & 7.65  & 4.30  & 8.73  \\
 $10^{-8}$ & 3.02  & 0.70  & 4.08  & - & 66.35  & 4.40  & 3.96  & 3.51  & 8.63  & 5.08  & 10.28  & 8.57  & 4.87  & 9.89  \\
 \hline \hline 
 sec. & 1.15 & 0.26 & 28.37 & 16.94 & 1292.33 & 17.88 & NaN & 1.50 & 2.25 & 1.36 & 44.50 & 21.90 & 2.23 & 135.19 \\ \hline 
\end{tabular}

}
\caption{Reached accuracy vs. number of FES/$n^2$ on $f_1$ with parameter $\ell=10^{0}$ in $n=20$ dimensions (mean over 31 independent runs). A dash `-' indicates that accuracy could not be reached within a budget of 200$n^2$ FES. See main text for further information.}
\label{tab:appendix_linexp0}
\end{table}
\begin{table}
\hspace{-1.8cm}
\scalebox{\sfactor}{
\begin{tabular}{| p{0.6cm} | P{1.3cm} | P{1.3cm} | P{1.3cm} | P{1.3cm} | P{1.3cm} | P{1.3cm} | P{1.3cm} |   P{1.3cm} | P{1.3cm} | P{1.3cm} | P{1.3cm} | P{1.3cm} | P{1.3cm} | P{1.3cm} |} 
\hline 
acc. & RP \newline ES & NEWUOA & IMFIL & N-M & Pattern & Nesterov RG & Nesterov Acc. & CMA-ES & V-RP (corr) ES & V-RP (corr) Exact & V-RP (corr) matlab & V-RP (store) ES & V-RP (store) Exact & V-RP (store) matlab \\ 
\hline  \hline 
 $10^{1}$ & 0.23  & 0.08  & 0.42  & 1.42  & 12.38  & 0.73  & 0.49  & 0.43  & 0.73  & 0.43  & 0.72  & 0.75  & 0.30  & 0.49  \\
 $10^{0}$ & 0.54  & 0.20  & 0.98  & 22.41  & 20.49  & 2.00  & 1.20  & 0.79  & 1.79  & 1.14  & 2.10  & 1.86  & 0.90  & 1.49  \\
 $10^{-1}$ & 0.91  & 0.27  & 1.57  & - & 30.73  & 3.53  & 2.26  & 1.17  & 2.78  & 1.91  & 3.51  & 2.89  & 1.51  & 2.66  \\
 $10^{-2}$ & 1.27  & 0.36  & 2.17  & - & 41.19  & 5.26  & 3.42  & 1.54  & 3.79  & 2.60  & 4.88  & 3.83  & 2.08  & 3.77  \\
 $10^{-3}$ & 1.61  & 0.43  & 2.97  & - & 51.40  & 7.01  & 4.38  & 1.92  & 4.70  & 3.21  & 6.19  & 4.79  & 2.60  & 4.84  \\
 $10^{-4}$ & 1.97  & 0.51  & 3.93  & - & 61.55  & 8.84  & 5.39  & 2.30  & 5.68  & 3.83  & 7.47  & 5.71  & 3.12  & 5.98  \\
 $10^{-5}$ & 2.31  & 0.59  & 4.93  & - & 72.39  & 10.65  & 6.53  & 2.69  & 6.63  & 4.45  & 8.63  & 6.62  & 3.65  & 7.12  \\
 $10^{-6}$ & 2.63  & 0.67  & 5.91  & - & 82.77  & 12.55  & 7.58  & 3.07  & 7.59  & 5.01  & 9.78  & 7.56  & 4.18  & 8.19  \\
 $10^{-7}$ & 2.98  & 0.74  & 6.75  & - & 93.21  & 14.48  & 8.64  & 3.45  & 8.60  & 5.59  & 10.99  & 8.47  & 4.69  & 9.39  \\
 $10^{-8}$ & 3.35  & 0.81  & 7.39  & - & 103.61  & 16.42  & 9.77  & 3.83  & 9.53  & 6.14  & 12.15  & 9.35  & 5.20  & 10.48  \\
 \hline \hline 
 sec. & 1.09 & 0.27 & 42.87 & 16.93 & 1652.45 & 18.72 & 22.82 & 1.64 & 2.26 & 1.47 & 32.35 & 23.23 & 2.60 & 131.40 \\ \hline 
\end{tabular}

}
\caption{Reached accuracy vs. number of FES/$n^2$ on $f_1$ with parameter $\ell=10^{1}$ in $n=20$ dimensions (mean over 31 independent runs). A dash `-' indicates that accuracy could not be reached within a budget of 200$n^2$ FES. See main text for further information.}
\label{tab:appendix_linexp1}
\end{table}
\begin{table}
\hspace{-1.8cm}
\scalebox{\sfactor}{
\begin{tabular}{| p{0.6cm} | P{1.3cm} | P{1.3cm} | P{1.3cm} | P{1.3cm} | P{1.3cm} | P{1.3cm} | P{1.3cm} |   P{1.3cm} | P{1.3cm} | P{1.3cm} | P{1.3cm} | P{1.3cm} | P{1.3cm} | P{1.3cm} |} 
\hline 
acc. & RP \newline ES & NEWUOA & IMFIL & N-M & Pattern & Nesterov RG & Nesterov Acc. & CMA-ES & V-RP (corr) ES & V-RP (corr) Exact & V-RP (corr) matlab & V-RP (store) ES & V-RP (store) Exact & V-RP (store) matlab \\ 
\hline  \hline 
 $10^{2}$ & 0.10  & 0.06  & 0.53  & 0.61  & 4.41  & 0.55  & 0.35  & 0.35  & 0.35  & 0.13  & 0.28  & 0.33  & 0.12  & 0.19  \\
 $10^{1}$ & 0.54  & 0.21  & 1.25  & 5.11  & 25.12  & 4.13  & 2.46  & 0.90  & 1.82  & 1.18  & 2.24  & 1.83  & 0.83  & 1.42  \\
 $10^{0}$ & 1.42  & 0.44  & 2.14  & 65.80  & 50.68  & 14.13  & 5.55  & 1.51  & 3.40  & 2.53  & 4.77  & 3.31  & 1.85  & 3.53  \\
 $10^{-1}$ & 2.50  & 0.70  & 3.20  & - & 78.55  & 28.72  & 9.23  & 2.15  & 4.75  & 3.98  & 7.39  & 4.21  & 2.93  & 5.55  \\
 $10^{-2}$ & 3.68  & 0.94  & 4.47  & - & 108.86  & 45.79  & 12.95  & 2.83  & 5.88  & 5.17  & 9.69  & 5.12  & 3.91  & 7.26  \\
 $10^{-3}$ & 4.91  & 1.17  & 5.91  & - & - & 64.04  & 16.96  & 3.41  & 6.98  & 6.26  & 11.70  & 5.95  & 4.68  & 8.67  \\
 $10^{-4}$ & 6.17  & 1.40  & 7.03  & - & - & 83.27  & 20.92  & 3.99  & 8.07  & 7.22  & 13.33  & 6.94  & 5.37  & 10.00  \\
 $10^{-5}$ & 7.42  & 1.62  & 7.55  & - & - & 102.76  & 24.94  & 4.55  & 8.99  & 8.05  & 14.77  & 7.89  & 5.99  & 11.19  \\
 $10^{-6}$ & 8.72  & 1.83  & 7.83  & - & - & 122.68  & 29.09  & 5.05  & 9.96  & 8.80  & 16.18  & 8.78  & 6.56  & 12.37  \\
 $10^{-7}$ & 10.10  & 2.04  & 8.07  & - & - & 142.71  & 33.35  & 5.56  & 10.88  & 9.48  & 17.55  & 9.68  & 7.14  & 13.54  \\
 $10^{-8}$ & 11.50  & 2.22  & 8.25  & - & - & 162.88  & 37.54  & 6.06  & 11.80  & 10.12  & 18.85  & 10.61  & 7.70  & 14.62  \\
 \hline \hline 
 sec. & 3.37 & 0.50 & 45.49 & 16.30 & 1863.95 & 17.80 & 23.71 & 2.30 & 2.67 & 2.02 & 43.35 & 28.89 & 10.28 & 383.55 \\ \hline 
\end{tabular}

}
\caption{Reached accuracy vs. number of FES/$n^2$ on $f_1$ with parameter $\ell=10^{2}$ in $n=20$ dimensions (mean over 31 independent runs). A dash `-' indicates that accuracy could not be reached within a budget of 200$n^2$ FES. See main text for further information.}
\label{tab:appendix_linexp2}
\end{table}
\begin{table}
\hspace{-1.8cm}
\scalebox{\sfactor}{
\begin{tabular}{| p{0.6cm} | P{1.3cm} | P{1.3cm} | P{1.3cm} | P{1.3cm} | P{1.3cm} | P{1.3cm} | P{1.3cm} |   P{1.3cm} | P{1.3cm} | P{1.3cm} | P{1.3cm} | P{1.3cm} | P{1.3cm} | P{1.3cm} |} 
\hline 
acc. & RP \newline ES & NEWUOA & IMFIL & N-M & Pattern & Nesterov RG & Nesterov Acc. & CMA-ES & V-RP (corr) ES & V-RP (corr) Exact & V-RP (corr) matlab & V-RP (store) ES & V-RP (store) Exact & V-RP (store) matlab \\ 
\hline  \hline 
 $10^{3}$ & 0.03  & 0.06  & 0.56  & 0.33  & 0.56  & 0.27  & 0.22  & 0.14  & 0.12  & 0.05  & 0.08  & 0.10  & 0.05  & 0.07  \\
 $10^{2}$ & 0.38  & 0.17  & 2.30  & 1.95  & 27.54  & 3.63  & 3.17  & 0.97  & 1.27  & 0.81  & 1.52  & 1.17  & 0.58  & 0.97  \\
 $10^{1}$ & 2.12  & 0.61  & 4.64  & 28.07  & - & 27.79  & 13.60  & 2.27  & 3.81  & 3.15  & 4.99  & 3.20  & 2.26  & 4.06  \\
 $10^{0}$ & 6.65  & 1.38  & 6.51  & - & - & 115.70  & 24.95  & 3.76  & 6.45  & 6.36  & 10.71  & 4.15  & 4.80  & 8.52  \\
 $10^{-1}$ & 13.08  & 2.18  & 8.59  & - & - & - & 39.27  & 5.06  & 8.38  & 8.97  & 15.30  & 5.05  & 5.77  & 10.64  \\
 $10^{-2}$ & 20.44  & 2.90  & 9.11  & - & - & - & 51.70  & 6.25  & 9.92  & 10.84  & 18.93  & 6.02  & 6.32  & 11.60  \\
 $10^{-3}$ & 28.33  & 3.73  & 9.42  & - & - & - & 65.28  & 7.29  & 11.22  & 12.38  & 21.99  & 6.94  & 6.86  & 12.61  \\
 $10^{-4}$ & 36.81  & 4.42  & 9.66  & - & - & - & 79.90  & 8.30  & 12.31  & 13.62  & 24.44  & 7.87  & 7.36  & 13.78  \\
 $10^{-5}$ & 45.38  & 5.14  & 9.85  & - & - & - & 92.56  & 9.13  & 13.33  & 14.74  & 26.33  & 8.78  & 7.88  & 14.89  \\
 $10^{-6}$ & 54.13  & 5.70  & 10.13  & - & - & - & 107.68  & 9.89  & 14.24  & 15.71  & 28.09  & 9.74  & 8.36  & 15.99  \\
 $10^{-7}$ & 63.07  & 6.24  & 10.41  & - & - & - & 121.35  & 10.59  & 15.23  & 16.61  & 29.74  & 10.67  & 8.86  & 17.06  \\
 $10^{-8}$ & 72.09  & 6.74  & 10.60  & - & - & - & 138.25  & 11.20  & 16.15  & 17.42  & 31.22  & 11.60  & 9.38  & 18.21  \\
 \hline \hline 
 sec. & 23.01 & 1.61 & 65.92 & 18.51 & 1843.56 & 19.05 & 21.96 & 4.28 & 4.02 & 3.09 & 48.35 & 31.62 & 19.01 & 596.88 \\ \hline 
\end{tabular}

}
\caption{Reached accuracy vs. number of FES/$n^2$ on $f_1$ with parameter $\ell=10^{3}$ in $n=20$ dimensions (mean over 31 independent runs). A dash `-' indicates that accuracy could not be reached within a budget of 200$n^2$ FES. See main text for further information.}
\label{tab:appendix_linexp3}
\end{table}
\begin{table}
\hspace{-1.8cm}
\scalebox{\sfactor}{
\begin{tabular}{| p{0.6cm} | P{1.3cm} | P{1.3cm} | P{1.3cm} | P{1.3cm} | P{1.3cm} | P{1.3cm} | P{1.3cm} |   P{1.3cm} | P{1.3cm} | P{1.3cm} | P{1.3cm} | P{1.3cm} | P{1.3cm} | P{1.3cm} |} 
\hline 
acc. & RP \newline ES & NEWUOA & IMFIL & N-M & Pattern & Nesterov RG & Nesterov Acc. & CMA-ES & V-RP (corr) ES & V-RP (corr) Exact & V-RP (corr) matlab & V-RP (store) ES & V-RP (store) Exact & V-RP (store) matlab \\ 
\hline  \hline 
 $10^{4}$ & 0.01  & 0.03  & 0.64  & 0.23  & 0.01  & 0.14  & 0.11  & 0.02  & 0.04  & 0.02  & 0.03  & 0.04  & 0.02  & 0.04  \\
 $10^{3}$ & 0.23  & 0.12  & 1.81  & 1.13  & 7.74  & 2.64  & 2.84  & 0.93  & 0.84  & 0.49  & 0.67  & 0.81  & 0.30  & 0.61  \\
 $10^{2}$ & 2.00  & 0.50  & 7.75  & 5.61  & - & 26.29  & 31.08  & 2.97  & 2.95  & 2.32  & 3.83  & 3.06  & 1.66  & 3.36  \\
 $10^{1}$ & 11.98  & 1.80  & 58.45  & 71.09  & - & - & 69.18  & 6.13  & 7.71  & 7.08  & 12.08  & 4.28  & 5.48  & 9.98  \\
 $10^{0}$ & 38.13  & 4.36  & - & - & - & - & 111.78  & 8.69  & 11.54  & 12.30  & 23.04  & 5.15  & 6.22  & 11.36  \\
 $10^{-1}$ & 79.47  & 7.32  & - & - & - & - & 160.50  & 10.91  & 14.12  & 16.49  & 30.71  & 6.06  & 6.70  & 12.36  \\
 $10^{-2}$ & 127.32  & 9.78  & - & - & - & - & - & 12.92  & 15.83  & 19.27  & 36.45  & 6.92  & 7.21  & 13.37  \\
 $10^{-3}$ & - & 12.06  & - & - & - & - & - & 14.31  & 17.16  & 21.34  & 39.90  & 7.83  & 7.74  & 14.39  \\
 $10^{-4}$ & - & 14.10  & - & - & - & - & - & 15.63  & 18.35  & 23.05  & 42.88  & 8.72  & 8.23  & 15.54  \\
 $10^{-5}$ & - & 15.80  & - & - & - & - & - & 16.66  & 19.40  & 24.38  & 45.28  & 9.65  & 8.76  & 16.72  \\
 $10^{-6}$ & - & 17.25  & - & - & - & - & - & 17.42  & 20.40  & 25.50  & 47.26  & 10.54  & 9.30  & 17.83  \\
 $10^{-7}$ & - & 18.60  & - & - & - & - & - & 18.01  & 21.36  & 26.51  & 48.93  & 11.49  & 9.81  & 18.92  \\
 $10^{-8}$ & - & 20.03  & - & - & - & - & - & 18.56  & 22.31  & 27.41  & 50.54  & 12.39  & 10.32  & 20.06  \\
 \hline \hline 
 sec. & 35.37 & 5.24 & 8588.05 & 17.96 & 1853.12 & 17.40 & 21.82 & 7.07 & 5.21 & 5.08 & 58.95 & 29.18 & 23.94 & 624.25 \\ \hline 
\end{tabular}

}
\caption{Reached accuracy vs. number of FES/$n^2$ on $f_1$ with parameter $\ell=10^{4}$ in $n=20$ dimensions (mean over 31 independent runs). A dash `-' indicates that accuracy could not be reached within a budget of 200$n^2$ FES. See main text for further information.}
\label{tab:appendix_linexp4}
\end{table}
\begin{table}
\hspace{-1.8cm}
\scalebox{\sfactor}{
\begin{tabular}{| p{0.6cm} | P{1.3cm} | P{1.3cm} | P{1.3cm} | P{1.3cm} | P{1.3cm} | P{1.3cm} | P{1.3cm} |   P{1.3cm} | P{1.3cm} | P{1.3cm} | P{1.3cm} | P{1.3cm} | P{1.3cm} | P{1.3cm} |} 
\hline 
acc. & RP \newline ES & NEWUOA & IMFIL & N-M & Pattern & Nesterov RG & Nesterov Acc. & CMA-ES & V-RP (corr) ES & V-RP (corr) Exact & V-RP (corr) matlab & V-RP (store) ES & V-RP (store) Exact & V-RP (store) matlab \\ 
\hline  \hline 
 $10^{5}$ & 0.01  & 0.01  & 0.76  & 0.14  & 0.01  & 0.06  & 0.06  & 0.01  & 0.02  & 0.02  & 0.02  & 0.02  & 0.01  & 0.02  \\
 $10^{4}$ & 0.16  & 0.09  & 1.59  & 0.86  & 2.87  & 2.01  & 2.21  & 0.71  & 0.62  & 0.27  & 0.44  & 0.60  & 0.20  & 0.33  \\
 $10^{3}$ & 1.19  & 0.33  & 5.95  & 2.83  & 59.55  & 20.91  & 79.53  & 2.87  & 2.38  & 1.59  & 2.79  & 2.19  & 1.30  & 2.58  \\
 $10^{2}$ & 11.48  & 1.44  & 54.07  & 18.63  & - & - & - & 7.18  & 7.16  & 6.79  & 11.99  & 4.77  & 5.56  & 11.23  \\
 $10^{1}$ & 88.03  & 5.25  & - & - & - & - & - & 12.53  & 14.00  & 15.21  & 29.17  & 5.74  & 7.58  & 14.94  \\
 $10^{0}$ & - & 14.35  & - & - & - & - & - & 16.35  & 18.34  & 22.97  & 44.82  & 6.66  & 8.16  & 16.07  \\
 $10^{-1}$ & - & 22.55  & - & - & - & - & - & 19.91  & 21.17  & 27.41  & 53.49  & 7.57  & 8.68  & 17.21  \\
 $10^{-2}$ & - & 29.21  & - & - & - & - & - & 21.87  & 22.87  & 30.26  & 58.76  & 8.52  & 9.20  & 18.31  \\
 $10^{-3}$ & - & 34.72  & - & - & - & - & - & 23.54  & 24.24  & 32.46  & 62.10  & 9.45  & 9.75  & 19.39  \\
 $10^{-4}$ & - & 38.96  & - & - & - & - & - & 24.50  & 25.42  & 34.22  & 64.89  & 10.34  & 10.25  & 20.53  \\
 $10^{-5}$ & - & 42.51  & - & - & - & - & - & 25.06  & 26.47  & 35.58  & 67.16  & 11.22  & 10.75  & 21.68  \\
 $10^{-6}$ & - & 45.76  & - & - & - & - & - & 25.52  & 27.47  & 36.76  & 69.11  & 12.11  & 11.26  & 22.88  \\
 $10^{-7}$ & - & 48.49  & - & - & - & - & - & 26.03  & 28.42  & 37.84  & 70.90  & 13.05  & 11.76  & 24.03  \\
 $10^{-8}$ & - & 51.14  & - & - & - & - & - & 26.45  & 29.37  & 38.76  & 72.66  & 14.01  & 12.33  & 25.13  \\
 \hline \hline 
 sec. & 34.46 & 50.67 & 8273.24 & 16.28 & 1863.09 & 19.17 & 21.96 & 10.08 & 6.69 & 7.15 & 66.13 & 36.11 & 26.35 & 569.45 \\ \hline 
\end{tabular}

}
\caption{Reached accuracy vs. number of FES/$n^2$ on $f_1$ with parameter $\ell=10^{5}$ in $n=20$ dimensions (mean over 31 independent runs). A dash `-' indicates that accuracy could not be reached within a budget of 200$n^2$ FES. See main text for further information.}
\label{tab:appendix_linexp5}
\end{table}
\begin{table}
\hspace{-1.8cm}
\scalebox{\sfactor}{
\begin{tabular}{| p{0.6cm} | P{1.3cm} | P{1.3cm} | P{1.3cm} | P{1.3cm} | P{1.3cm} | P{1.3cm} | P{1.3cm} |   P{1.3cm} | P{1.3cm} | P{1.3cm} | P{1.3cm} | P{1.3cm} | P{1.3cm} | P{1.3cm} |} 
\hline 
acc. & RP \newline ES & NEWUOA & IMFIL & N-M & Pattern & Nesterov RG & Nesterov Acc. & CMA-ES & V-RP (corr) ES & V-RP (corr) Exact & V-RP (corr) matlab & V-RP (store) ES & V-RP (store) Exact & V-RP (store) matlab \\ 
\hline  \hline 
 $10^{6}$ & 0.01  & 0.01  & 0.87  & 0.04  & 0.01  & 0.01  & 0.01  & 0.01  & 0.01  & 0.01  & 0.02  & 0.02  & 0.01  & 0.02  \\
 $10^{5}$ & 0.10  & 0.08  & 1.57  & 0.65  & 2.37  & 1.76  & 1.57  & 0.45  & 0.60  & 0.19  & 0.26  & 0.56  & 0.11  & 0.23  \\
 $10^{4}$ & 0.77  & 0.27  & 4.51  & 1.89  & 31.90  & 17.38  & - & 2.20  & 1.93  & 1.07  & 2.26  & 1.98  & 0.80  & 1.79  \\
 $10^{3}$ & 8.45  & 1.15  & 42.37  & 5.98  & - & 170.48  & - & 6.29  & 6.40  & 4.41  & 10.98  & 5.08  & 4.26  & 10.04  \\
 $10^{2}$ & 79.40  & 5.11  & - & 35.85  & - & - & - & 13.25  & 13.16  & 12.54  & 33.67  & 7.20  & 9.16  & 20.72  \\
 $10^{1}$ & - & 19.53  & - & - & - & - & - & 21.68  & 20.56  & 24.92  & 58.73  & 8.21  & 10.67  & 24.05  \\
 $10^{0}$ & - & 37.01  & - & - & - & - & - & 26.47  & 25.09  & 33.63  & 74.31  & 9.09  & 11.38  & 25.24  \\
 $10^{-1}$ & - & 52.61  & - & - & - & - & - & 29.18  & 27.85  & 38.78  & 83.25  & 9.99  & 11.92  & 26.35  \\
 $10^{-2}$ & - & 65.85  & - & - & - & - & - & 31.09  & 29.72  & 41.81  & 88.14  & 10.87  & 12.41  & 27.41  \\
 $10^{-3}$ & - & 74.92  & - & - & - & - & - & 31.90  & 31.25  & 43.96  & 91.83  & 11.80  & 12.95  & 28.41  \\
 $10^{-4}$ & - & 82.32  & - & - & - & - & - & 32.64  & 32.40  & 45.62  & 94.67  & 12.70  & 13.45  & 29.53  \\
 $10^{-5}$ & - & 89.27  & - & - & - & - & - & 33.08  & 33.45  & 46.93  & 97.09  & 13.67  & 13.97  & 30.68  \\
 $10^{-6}$ & - & 95.14  & - & - & - & - & - & 33.48  & 34.41  & 48.12  & 99.09  & 14.60  & 14.49  & 31.81  \\
 $10^{-7}$ & - & 100.65  & - & - & - & - & - & 34.01  & 35.32  & 49.17  & 100.86  & 15.57  & 14.99  & 32.99  \\
 $10^{-8}$ & - & 105.19  & - & - & - & - & - & 34.41  & 36.34  & 50.05  & 102.39  & 16.51  & 15.52  & 34.21  \\
 \hline \hline 
 sec. & 30.90 & 194.63 & 7987.02 & 16.29 & 1888.71 & 17.42 & 22.54 & 12.99 & 8.47 & 9.21 & 76.31 & 40.34 & 29.91 & 547.15 \\ \hline 
\end{tabular}

}
\caption{Reached accuracy vs. number of FES/$n^2$ on $f_1$ with parameter $\ell=10^{6}$ in $n=20$ dimensions (mean over 31 independent runs). A dash `-' indicates that accuracy could not be reached within a budget of 200$n^2$ FES. See main text for further information.}
\label{tab:appendix_linexp6}
\end{table}
\begin{table}
\hspace{-1.8cm}
\scalebox{\sfactor}{
\begin{tabular}{| p{0.6cm} | P{1.3cm} | P{1.3cm} | P{1.3cm} | P{1.3cm} | P{1.3cm} | P{1.3cm} | P{1.3cm} |   P{1.3cm} | P{1.3cm} | P{1.3cm} | P{1.3cm} | P{1.3cm} | P{1.3cm} | P{1.3cm} |} 
\hline 
acc. & RP \newline ES & NEWUOA & IMFIL & N-M & Pattern & Nesterov RG & Nesterov Acc. & CMA-ES & V-RP (corr) ES & V-RP (corr) Exact & V-RP (corr) matlab & V-RP (store) ES & V-RP (store) Exact & V-RP (store) matlab \\ 
\hline  \hline 
 $10^{6}$ & 0.08  & 0.07  & 1.58  & 0.56  & 0.95  & 1.44  & 0.96  & 0.30  & 0.48  & 0.11  & 0.13  & 0.52  & 0.12  & 0.15  \\
 $10^{5}$ & 0.54  & 0.23  & 3.85  & 1.52  & 17.25  & 14.66  & - & 1.84  & 1.48  & 0.81  & 1.53  & 1.53  & 0.59  & 1.36  \\
 $10^{4}$ & 6.50  & 0.94  & 21.60  & 4.43  & - & 142.72  & - & 5.64  & 4.94  & 3.52  & 6.83  & 4.53  & 3.14  & 7.41  \\
 $10^{3}$ & 66.49  & 4.03  & - & 17.45  & - & - & - & 11.49  & 12.13  & 12.30  & 28.55  & 8.04  & 9.57  & 22.97  \\
 $10^{2}$ & - & 17.57  & - & 67.45  & - & - & - & 20.94  & 19.79  & 24.87  & 59.06  & 9.84  & 12.91  & 32.27  \\
 $10^{1}$ & - & 42.19  & - & - & - & - & - & 30.81  & 26.41  & 37.43  & 91.47  & 10.91  & 14.50  & 36.40  \\
 $10^{0}$ & - & 82.48  & - & - & - & - & - & 36.20  & 31.75  & 46.10  & 110.00  & 11.89  & 15.19  & 37.98  \\
 $10^{-1}$ & - & 113.01  & - & - & - & - & - & 38.59  & 35.23  & 50.36  & 120.59  & 12.81  & 15.72  & 39.09  \\
 $10^{-2}$ & - & 133.34  & - & - & - & - & - & 40.11  & 37.23  & 53.69  & 126.24  & 13.78  & 16.25  & 40.10  \\
 $10^{-3}$ & - & 149.80  & - & - & - & - & - & 40.80  & 38.63  & 56.18  & 129.87  & 14.74  & 16.73  & 41.09  \\
 $10^{-4}$ & - & 165.02  & - & - & - & - & - & 41.22  & 39.90  & 57.91  & 132.38  & 15.65  & 17.27  & 42.16  \\
 $10^{-5}$ & - & 177.24  & - & - & - & - & - & 41.62  & 40.92  & 59.22  & 134.52  & 16.56  & 17.81  & 43.24  \\
 $10^{-6}$ & - & - & - & - & - & - & - & 42.01  & 41.92  & 60.42  & 136.33  & 17.49  & 18.31  & 44.39  \\
 $10^{-7}$ & - & - & - & - & - & - & - & 42.41  & 42.92  & 61.56  & 137.85  & 18.43  & 18.88  & 45.66  \\
 $10^{-8}$ & - & - & - & - & - & - & - & 42.82  & 43.94  & 62.44  & 139.45  & 19.40  & 19.44  & 46.85  \\
 \hline \hline 
 sec. & 30.98 & 447.57 & 7705.72 & 17.72 & 1861.84 & 18.08 & 22.77 & 16.30 & 10.45 & 11.96 & 74.15 & 45.21 & 34.37 & 519.52 \\ \hline 
\end{tabular}

}
\caption{Reached accuracy vs. number of FES/$n^2$ on $f_1$ with parameter $\ell=10^{7}$ in $n=20$ dimensions (mean over 31 independent runs). A dash `-' indicates that accuracy could not be reached within a budget of 200$n^2$ FES. See main text for further information.}
\label{tab:appendix_linexp7}
\end{table}

\section{RHE: llustrative numerical example}
\label{sec:numericssec5}
We now illustrate the typical convergence behavior of Variable Metric Random Pursuit on the challenging convex quadratic function $f_3$, introduced in Section~\ref{sec:benchfun}. This function has two different scales that need to be learned. We use parameter $\ell = 10^7$. The ratio of largest to smallest eigenvalue of the Hessian (i.e. the condition number) is $10^7$, and the global minimum of $f_3$ is at $\xx^*= \0_n$ (where $\0_n$ is the all-zeros vector) with $f_3(\xx^*)=0$. We conduct 51 runs of V-RP in $n=20$ dimensions. The initial conditions are $\xx_0=(1, \dots, 1, 1/\sqrt{\ell}, \dots, 1/\sqrt{\ell})^T$, $B_0=\frac{1}{2} \ell \cdot I_n$. The two VM update schemes \texttt{updateHessCorr} and \texttt{updateHessStore} (see Fig.~\ref{fig:algonew}) are tested with the setting $\epsilon=1$ for both schemes. The \texttt{updateHessStore} scheme reuses samples from the storage $S$ in every $n$-th iteration. We here report the evolution of the mean, maximum,  and minimum function value vs. number of iterations (\#ITS). We also calculate and report the evolution of the derived convergence factor $\hat{\varrho}$ from Thm.~\ref{thm:global_strong}. 

\begin{figure}[ht]
\centering
\includegraphics[width=.95\textwidth]{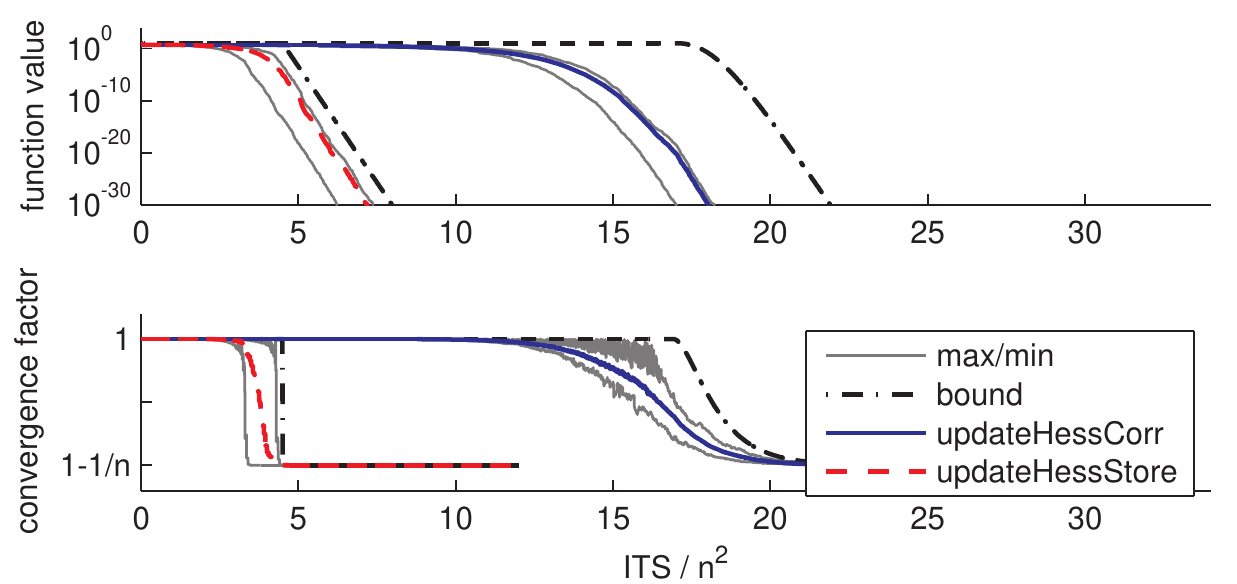}
\caption{Convergence of V-RP on $f_3$ for \texttt{updateHessCorr} (blue/solid) and \texttt{updateHessStore} (red/dashed).  Parameter $\ell = 10^7$ in $n=20$ dimensions. Upper panel:  Mean and max/min (grey) function values vs. \#ITS over 51 runs. Lower Panel: Mean and max/min (grey) convergence factor $\hat{\varrho}$ vs. \#ITS over 51 runs. Respective upper bounds (black/dash-dotted). See main text for further information.}
\label{fig:f1}
\end{figure}

On quadratic functions, a typical V-RP optimization trajectory (see ~\cite{Leventhal:2011,Stich:2012a} for several examples) shows three distinct phases of convergence in function values: (i) a first short phase of rapid improvement, (ii) a metric learning phase with only marginal progress in function decrease, and (iii) a final rapid decrease in function value. In the present experiments we chose the initial iterate $\xx_0$ such as to minimize the first phase. This allows a clearer quantification of the length of the adaptation phase. We see that the adaptation phase lasts for roughly 5$n^2$ iterations in case of \texttt{updateHessStore} and 15-18$n^2$ iterations for \texttt{updateHessCorr}. We also visualize the derived upper bounds on the convergence factor (see Remark~\ref{rem:markov}) in the lower panel of Fig.~\ref{fig:f1}. For \texttt{updateHessCorr} the curve is plotted using $b=1$, and for \texttt{updateHessStore} using $\tilde{c} = 1.5$. We see that the shape of both curves resembles the observed data. However, in both cases the theoretical bounds overestimate the empirically observed curves (``shifted'' to the right). In the upper panel of Fig.~\ref{fig:f1} we depict the theoretical derived upper bound on the function value (Thm.~\ref{thm:global_strong}). For \texttt{updateHessStore} the shape of this curve well matches the observed convergence. For \texttt{updateHessCorr} we see that the empirically observed phase transition between phase (ii) and (iii) occurs more smoothly than predicted by the theoretical bound.

\begin{figure}[ht]
\centering
\includegraphics[width=.95\textwidth]{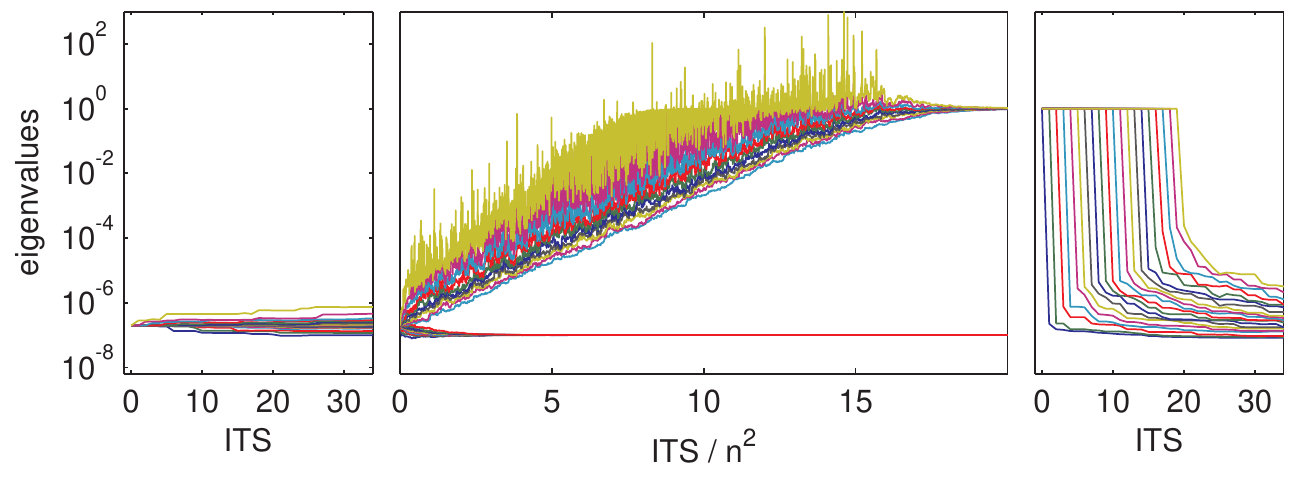}
\caption{Evolution of the spectrum of $\Sigma = B_k^{-1}$ for VM scheme \texttt{updateHessCorr} on $f_3$. Eigenvalues vs. \# ITS for 1 run. Parameter $\ell = 10^7$ in $n=20$ dimensions. Left two panels with initial setting $B_0 = \frac{\ell}{2} I_n$, right panel with $B_0 = I_n$. See main text for further information.}
\label{fig:f1hess}
\end{figure}

We also illustrate the evolution of the spectrum of the estimated inverse Hessian $\Sigma = B_k^{-1}$ in Fig.~\ref{fig:f1hess} for one run with update scheme \texttt{updateHessCorr}. At the beginning all eigenvalues are close to $\frac{2}{\ell}$, as $B_0^{-1} = \frac{2}{\ell} I_n$ (left panel). Then, about half of the eigenvalues start to increase up to 1, the other half decreases to $\frac{1}{\ell}$. We see that the large eigenvalues of $\Sigma$ (or correspondingly the small eigenvalues of $H$) are more difficult to approximate. This takes up to 16-18$n^2$ iterations. In the right panel we depicted another run with initial matrix $\Sigma = B_0^{-1} = I_n$. At the beginning all eigenvalues are equal to 1. Due to the nature of the VM update scheme (rank one updates), at most one eigenvalue can become different from 1 in every iteration. Thus it takes exactly $n=20$ iterations until all eigenvalues are between $10^{-7}$ and $10^{-5}$. From this moment, the situation is similar to the experiment in the left two panels with $B_0 = \frac{\ell}{2} I_n$.
\end{document}